\documentclass{siamart1116}
\usepackage{amsmath}
\usepackage{amsfonts}
\usepackage{algorithm}
\usepackage{algorithmic}
\usepackage{url}
\usepackage{subfigure}
\usepackage{enumitem}
\usepackage{titlesec}
\usepackage[titletoc]{appendix}
\usepackage{mathtools}

\def\noprint#1{}

\newtheorem{assumption}{Assumption}

\newcommand{\lambdamin}{\lambda_{\mbox{\rm\scriptsize{min}}}}

\newcommand{\R}{\mathbb{R}}

\newcommand{\cK}{{\cal K}}
\newcommand{\eps}{\epsilon}
\newcommand{\epsg}{\epsilon_g}
\newcommand{\epsH}{\epsilon_H}
\newcommand{\flow}{f_{\mbox{\rm\scriptsize low}}}

\newcommand{\csol}{c_{\mbox{\rm\scriptsize sol}}}
\newcommand{\jsol}{j_{\mbox{\rm\scriptsize sol}}}
\newcommand{\cnc}{c_{\mbox{\rm\scriptsize nc}}}
\newcommand{\call}{c_{\mbox{\rm\scriptsize all}}}
\newcommand{\Cmeo}{\mathcal{C}_{\mathrm{meo}}}

\newcommand{\jnc}{j_{\mbox{\rm\scriptsize nc}}}

\newcommand{\bH}{\bar{H}}

\newcommand{\cO}{{\cal O}}
\newcommand{\tcO}{\tilde{\mathcal O}}
\newcommand{\bX}{\bar{X}}

\newcommand{\hIcg}{\hat{J}}

\newcommand{\diag}{\mbox{\rm diag }}
\DeclarePairedDelimiter{\ceil}{\lceil}{\rceil}

\def\sjwresolved#1{}
\def\monresolved#1{}
\newcommand{\refer}[1]{#1}
\newcommand{\swmodify}[1]{#1}

\def\tto{\;{\lower 1pt \hbox{$\rightarrow$}}\kern -10pt
           \hbox{\raise 2.8pt \hbox{$\rightarrow$}}\;}

\title{A Log-Barrier Newton-CG Method 
for Bound Constrained Optimization with Complexity Guarantees%
  \thanks{Version of \today. Research supported by  NSF Awards IIS-1447449, 1628384, 1634597,
and 1740707; AFOSR Award FA9550-13-1-0138; Subcontracts 3F-30222 and 8F-30039
from Argonne National Laboratory; and Award N660011824020 from the DARPA Lagrange Program.}}

\author{Michael O'Neill and Stephen J. Wright}

\begin{document}

\maketitle

\begin{abstract}
We describe an algorithm based on a logarithmic barrier function,
Newton's method, and linear conjugate gradients that seeks an
approximate minimizer of a smooth function over the nonnegative
orthant. We develop a bound on the complexity of the approach, stated
in terms of the required accuracy and the cost of a single gradient
evaluation of the objective function and/or a matrix-vector
multiplication involving the Hessian of the objective. The approach
can be implemented without explicit calculation or storage of the
Hessian.
\end{abstract}



\section{Introduction} 
\label{sec:intro}

We consider the following constrained optimization problem:
\begin{equation} \label{eq:fdef}
\min f(x) \quad \mbox{subject to $x \geq 0$,}
\end{equation}
where $f : \R^n \rightarrow \R$ is a nonconvex function, twice
uniformly Lipschitz continuously differentiable in the interior of the
nonnegative orthant.  We assume that explicit storage of the Hessian
$\nabla^2 f(x)$ for $x>0$ is undesirable, but that Hessian-vector
products of the form $\nabla^2 f(x) v$ can be computed \swmodify{at
  any $x>0$} for arbitrary vectors $v$. Computational differentiation
techniques \cite{AGriewank_AWalther_2008} can be used to evaluate such
products at a cost that is a small multiple of the cost of evaluation
of the gradient $\nabla f$.

The problem \eqref{eq:fdef} is well studied, with numerous algorithms
being proposed over the years, based on such strategies as active set,
gradient projection, and Newton's method. Other possible approaches
include interior-point and barrier methods, which generate iterates
that remain strictly feasible. The primal log-barrier method minimizes
the log-barrier function
\begin{equation} \label{eq:logbarrierdef}
\phi_\mu(x) = f(x) - \mu \sum_{i=1}^n \log(x_i),
\end{equation}
for some decreasing sequence of positive scalars $\mu$
\cite{AVFiacco_GPMcCormick_1990}. The function $\phi_\mu$ can be
minimized using Newton's method with a line search strategy that
maintains strict positivity of the components of $x$ as well as
ensuring sufficient decrease at each iteration.


Our goal in this paper is to design and analyze a method with
attractive worst-case complexity guarantees comparable to those that
have been attained recently for unconstrained minimization of smooth
nonconvex functions. The algorithm we describe in this paper combines
the primal log-barrier formulation \eqref{eq:logbarrierdef} with the
Newton-Conjugate-Gradient (``Newton-CG'') algorithm of
\cite{CWRoyer_MONeill_SJWright_2019}.  We minimize the log-barrier
function $\phi_\mu$ for only a single value of $\mu$, chosen
judiciously to ensure that its approximate minimizer coincides with an
approximate solution to \eqref{eq:fdef} that satisfies our accuracy
criteria. The Newton-CG method applied to $\phi_\mu$ uses a
safeguarded version of the linear CG method to minimize a slightly
damped second-order Taylor series approximation of $\phi_\mu$ at each
iteration. In contrast to its application to unconstrained
optimization, the linear system is preconditioned to control the norm
of its coefficient matrix to ensure that the number of CG iterations
is bounded by a quantity that depends on the accuracy of the desired
solution.  The safeguarded CG method monitors its iterates for
evidence of indefiniteness in the Hessian, and outputs a direction of
negative curvature for this matrix if indefiniteness is detected. If
no indefiniteness is detected, this CG procedure finds an approximate
Newton step. In either case, we do a backtracking line search along
the chosen direction, and show that the decrease in $\phi_\mu$
\swmodify{at each step} is sufficient to place an overall bound on the
number of iterations, allowing worst-case complexity results to be
proved.

Although practical efficiency of the method is not our
main concern in this paper, we note that our method is a
``long-step'' interior-point method, of the kind that has been
useful in other settings.

The rest of this paper is organized as follows.
Section~\ref{sec:relatedworks} reviews related work, \swmodify{ puts
  our paper in context, and outlines our main result}. In
Section~\ref{sec:approxKKT} we derive a first- and second-order
approximate optimality condition for (\ref{eq:fdef}).
Section~\ref{sec:lbncg} describes our log-barrier Newton-CG algorithm,
while Section~\ref{sec:wcc} presents the worst-case complexity
analysis for the first- and second-order approximate KKT
conditions. Some conclusions appear in Section~\ref{sec:discussion}.

{\em Assumptions, Background, Notation.}  We assume the following
throughout, concerning smoothness and boundedness of $f$.
\refer{\begin{assumption} \label{assum:fC22}
The function $f$ is twice uniformly Lipschitz continuously differentiable
on an open neighborhood of the path of the iterates and trial points.
We denote by $L_g$ the Lipschtiz constant
for $\nabla f$ and $L_H$ the Lipschitz constant for $\nabla^2 f$
on this set.
\end{assumption}
\begin{assumption} \label{assum:flow}
The function $f$ is bounded below by $\flow$.
\end{assumption}
\begin{assumption} \label{assum:boundedgH}
The iterates $\{x^k\}$ satisfy,
\[
\|\nabla f(x^k)\| \leq U_g,
\quad \|\nabla^2 f(x^k)\| \leq U_H,
\]
for some scalars $U_g > 0$ and
$U_H > 0$.
\end{assumption}}
(Here and throughout we use $\|\cdot\|$ to denote the Euclidean norm,
or its induced norm on matrices.)  We observe that $U_H$ is a
Lipschitz constant for the gradient of $f$.

\swmodify{For any $x$ and $y$ such that Assumption \ref{assum:fC22} is
  satisfied}, we have
\begin{align} \label{eq:fliphessian}
&\|\nabla f(y) - \nabla f(x) - \nabla^2 f(x) (y-x)\| \leq \frac12 L_H \|x-y\|^2, \\
&f(y) \leq f(x) + \nabla f(x)^\top(y - x) + \frac12 (y - x)^\top \nabla^2 f(x) (y - x)
+ \frac{1}{6} L_H \|x - y\|^3. \label{eq:f2smooth}
\end{align}

Order notation $\cO$ is used in its usual sense, whereas $\tcO$
represents $\cO$ with logarithmic \swmodify{factors} omitted.

We define $e = (1, \dotsc,1)^\top$ to be the vector of ones and $e_i
= (0,\dotsc,0,1,0,\dotsc,0)^\top$ to be the unit vector with 1 as the
$i$th component and zeros elsewhere. The $i$th component of a vector
$v$ is denoted by $v_i$ or $[v]_i$. Given a vector $x \in \R^n_+$
(where $\R^n_+$ is the nonnegative orthant), we denote by $X$ the
diagonal matrix formed by the components of $x$, by $\bar{x}$ the
vector whose components are $\min(x_i,1)$,\footnote{We use a threshold
  of $1$ for clarity of presentation. Any other positive value could
  be used instead, with minimal effect on the results.} and by
$\bar{X}$ the diagonal matrix formed from $\bar{x}$. That is,
\begin{equation} \label{eq:Xdef}
  X = \diag(x_1,x_2,\dotsc,x_n), \quad
  \bar{x} = \min (x,e), \quad
  \bX = \diag(\bar{x}_1,\bar{x}_2,\dotsc,\bar{x}_n).
  \end{equation}

Our algorithm seeks a point $x$ satisfying the following approximate
optimality conditions for (\ref{eq:fdef}):
\begin{subequations} \label{eq:epsKKT}
\begin{align}
x &> 0, \label{eq:xstrictfeas} \\
\nabla f(x) &\geq -\epsg e, \label{eq:epsgradpos} \\
\|\bX \nabla f(x) \|_\infty &\leq \epsg, \label{eq:epscompliment} \\
\bX \nabla^2 f(x) \bX &\succeq -\epsH I \label{eq:epspsd},
\end{align}
\end{subequations}
for small positive tolerances $\epsg$ and $\epsH$. The conditions
(\ref{eq:epscompliment}) and (\ref{eq:epspsd}) differ from the scaled
gradient and Hessian conditions used elsewhere, through the
substitution of the bounded matrix $\bX$ for $X$. The theoretical
basis for these conditions as well as their relation to those used in
previous works is presented in Section~\ref{sec:approxKKT}.

\section{Related Work} \label{sec:relatedworks}

There is considerable recent work on algorithms for {\em
  unconstrained} smooth nonconvex optimization that have optimal
worst-case iteration complexity for finding points that satisfy
approximate first- and second-order optimality conditions.  When
applied to twice Lipschitz continuously differentiable functions,
classical Newton-trust-region schemes
\cite{ARConn_NIMGould_PhLToint_2000} require at most
$\cO\left(\max\left\{\epsg^{-2} \epsH^{-1}, \epsH^{-3}
\right\}\right)$ iterations \cite{CCartis_NIMGould_PLToint_2012a} to
find a point satisfying
\begin{equation} \label{eq:unconstrained2on}
\|\nabla f(x)\| \leq \epsg\; \mbox{ and } \;
\lambdamin (\nabla^2 f(x)) \geq -\epsH.
\end{equation}
For this class of problems, the optimal iteration complexity for
finding a \refer{second-order optimal point  is
  $\cO\left(\max\left\{\epsg^{-3/2},\epsH^{-3}\right\}\right)$}
\cite{YCarmon_JCDuchi_OHinder_ASidford_2017b,CCartis_NIMGould_PhLToint_2011c,CCartis_NIMGould_PhLToint_2018a}.
This iteration complexity was first achieved by cubic regularization
of Newton's method \cite{YuNesterov_BTPolyak_2006}.
Numerous other algorithms have also been proposed that match this
iteration bound; see for example
\cite{EGBirgin_JMMartinez_2017,CCartis_NIMGould_PhLToint_2011a,FECurtis_DPRobinson_MSamadi_2017,FECurtis_DPRobinson_MSamadi_2018a,JMMartinez_MRaydan_2017}.

Some works also account for the computational cost of each iteration,
thus yielding a bound on the overall computational complexity.
\refer{A number of works have focused on efficiently computing a
  solution to the cubically regularized subproblem, either through
  direct matrix factorization techniques
  \cite{EGBirgin_JMMartinez_2019a, CCartis_NIMGould_PhLToint_2011a,
    YuNesterov_BTPolyak_2006} and/or Krylov subspace based methods
  \cite{CCartis_NIMGould_PhLToint_2011a, JPDussault_DOrban_2015}.
  These approaches yield a worst case operational complexity of
  $\cO(n \epsg^{-3/2})$ when $\epsH = \epsg^{1/2}$.}
Two independently proposed algorithms, respectively based on
adapting accelerated gradient to the nonconvex setting
\cite{YCarmon_JCDuchi_OHinder_ASidford_2018} and approximately solving
the cubic regularization subproblem
\cite{NAgarwal_ZAllenZhu_BBullins_EHazan_TMa_2017}, require $\tcO
(\epsg^{-7/4})$ operations (with high probability, showing dependency
only on $\epsg$) to find a point $x$ that satisfies
(\ref{eq:unconstrained2on}) when $\epsH = \epsg^{1/2}$.  The
difference of a factor of $\epsg^{-1/4}$ with the iteration complexity
bounds arises from the cost of computing a negative curvature
direction of $\nabla^2 f(x_k)$ and/or the cost of solving a linear
system. The probabilistic nature of the bound is due to the
introduction of randomness in the curvature estimation process.
A complexity bound of the same type was also established for
a variant of accelerated gradient based only on gradient calculations,
that periodically adds a random perturbation to the iterate when the
gradient norm is small~\cite{CJin_PNetrapalli_MIJordan_2018}.

In another line of work, \cite{CWRoyer_SJWright_2018} developed a
damped Newton algorithm which inexactly minimizes the Newton system by
the method of conjugate gradients and requires at most $\tcO
(\min\{n\epsg^{-3/2}, \epsg^{-7/4}\})$ operations to satisfy
(\ref{eq:unconstrained2on}), to high probability. For purposes of
computational complexity, this paper defines the unit of computation
to be one Hessian-vector product {\em or} one gradient evaluation. We
also adopt this definition here; it relies implicitly on the
observation from computational / algorithmic differentiation
\cite{AGriewank_AWalther_2008} that these two operations differ in
cost only by a modest factor, independent of the dimension $n$. In a
followup to \cite{CWRoyer_SJWright_2018}, the paper
\cite{CWRoyer_MONeill_SJWright_2019} built on techniques from
\cite{YCarmon_JCDuchi_OHinder_ASidford_2017a} to create a modified CG
method to solve the Newton system.  This algorithm, which is a
foundation of the method described in this paper, again finds a point
satisfying \eqref{eq:unconstrained2on} in $\tcO (\min\{n\epsg^{-3/2},
\epsg^{-7/4}\})$ operations, to high probability, and requires the
same number of operations to find an approximate first-order critical
point {\em deterministically}.

A number of algorithms have also been proposed for {\em constrained}
optimization problems that require at most $\cO(\max \{\epsg^{-3/2},
\epsH^{-3} \})$ iterations to find a point which satisfies some
first-order (and sometimes second-order) optimality
conditions. Although the optimality conditions vary between papers,
\swmodify{the works
  \cite{CCartis_NIMGould_PLToint_2012b,CCartis_NIMGould_PLToint_2015}
  achieve this iteration complexity bound for some first-order
  optimality condition by solving a constrained cubic regularization
  subproblem at each iteration.} \refer{These approaches have been
  greatly simplified in recent times for problems involving
  ``inexpensive'' convex constraints
  \cite{CCartis_NIMGould_PhLToint_2018b,CCartis_NIMGould_PhLToint_2019a}}.
A different proposal finds a first-order point in $\cO (\epsg^{-3/2}
)$ iterations for linear equality and bound constraints through the
use of an active set method \cite{EGBirgin_JMMartinez_2018}.  When
optimizing on a single face of the polytope, this method also uses a
cubic regularization model. However, these papers do not account for
the cost of solving the subproblem at each iteration, noting either
that this subproblem may be NP-hard, or suggesting that a simple
first-order, gradient-based method can solve it reliably.
\refer{Many other methods have been proposed for constrained
  optimization which have good worst-case iteration complexity results, such
  as two-phase methods
  \cite{EGBirgin_JMMartinez_SASantos_PLToint_2016,
    CCartis_NIMGould_PhLToint_2019b,
    FECurtis_DPRobinson_MSamadi_2018b}, an interior-point method
  \cite{OHinder_YYe_2018}, and augmented Lagrangian methods
  \cite{EGBirgin_JMMartinez_2019b,
    GNGrapiglia_YYuan_2019,YXie_SJWright_2019}.}

Turning to our bound-constrained problem \eqref{eq:fdef}, a
second-order interior-point method was proposed in
\cite{WBian_XChen_YYe_2015}. This method minimizes a preconditioned
second-order trust-region model at each iteration and finds a point
satisfying approximate second-order conditions in at most $\cO
(\epsg^{-3/2})$ iterations when $\epsH = \epsg^{1/2}$.  However, the
first-order conditions are strictly weaker than those used in the
current work as they consist only of feasibility of $x$ along with a
scaled gradient condition that is an ``unbounded'' version of
(\ref{eq:epscompliment}) in which $\bX$ is replaced by $X$.
\refer{Without additional assumptions on $f$, the absence of condition
  (\ref{eq:epsgradpos}) in the optimality conditions implies that
  sequences of (strictly feasible) points that satisfy the scaled
  gradient condition may not approach KKT points as $\epsg$ approaches
  0; see \cite[Section~2]{GHaeser_HLiu_YYe_2018} for a discussion of
  this issue.} Our approximate optimality conditions \eqref{eq:epsKKT}
here do not suffer from these issues, as we show in
Section~\ref{sec:approxKKT}.  In a follow up to
\cite{WBian_XChen_YYe_2015}, an interior-point method for linear
equality and bound constraints was described in
\cite{GHaeser_HLiu_YYe_2018}. This method, which also achieves an
iteration complexity of $\cO (\epsg^{-3/2})$ (when $\epsH =
\epsg^{1/2}$), applies a constrained second-order trust-region
algorithm to the log-barrier function, with a (potentially) small
trust-region radius. \swmodify{The authors of
  \cite{GHaeser_HLiu_YYe_2018} were more interested in iteration
  complexity than computational complexity, but we note that each of
  their subproblems requires evaluation of the Hessian (which in the
  worst case requires evaluation of $n$ Hessian-vector products, where
  the latter is one of our units of computational complexity),
  together with $\tcO(n^3)$ floating point operations associated with performing
  a bisection scheme to solve the subproblem. These considerations
  suggest an overall worst-case computational complexity of at least
  $\cO(n \epsg^{-3/2})$ for the algorithm of
  \cite{GHaeser_HLiu_YYe_2018}.}



In this paper, we adapt the Newton-CG method of
\cite{CWRoyer_MONeill_SJWright_2019} for unconstrained optimization to
the problem of minimizing the primal log-barrier function
\eqref{eq:logbarrierdef}, for a small, fixed value of $\mu$.  We
target the optimality conditions (\ref{eq:epsKKT}), which avoid
enforcing tighter conditions on Hessian and gradient components that
correspond to components of $x$ that are far from zero at optimality.
This change allows us to solve a preconditioned Newton system of
linear equations at each iteration in which the norm of the matrix can
be bounded by a constant independent of iteration number.  The Capped
CG method developed in \cite{CWRoyer_MONeill_SJWright_2019} is used to
solve this system, returning a useful search direction in a reasonable
number of iterations.  \sjwresolved{We needed a much much clearer
  statement of our ultimate results and contribution here and
  later. They were buried in the final few results and impossible to
  find except for the most dedicated reader. I have tried to dig them
  out, here and there, but may not be sufficiently clear yet. I am
  trying to remember to highlight my changes in this regard.}  When
$\epsH = \epsg^{1/2}$, our algorithm finds a point satisfying
(\ref{eq:epsKKT}) in \swmodify{$\tcO(n \epsg^{-1/2} + \epsg^{-3/2})$
  iterations (Theorem~\ref{thm:wccits}).  The computational
  complexity, in terms of gradient evaluations/Hessian vector
  products, is $\tcO(n \epsg^{-3/4} + \epsg^{-7/4})$ for large values
  of $n$, and $\tcO(n \epsg^{-3/2})$ for smaller $n$; see
  Corollary~\ref{coro:wcc2Hv} and the comments following this
  result. The appearance of $n$ in our complexity expressions is an
  apparently unavoidable consequence of using log-barrier methodology,
  along with making the mildest possible assumptions on the problem
  \eqref{eq:fdef} and the algorithm.  \monresolved{Slightly modified
    this sentence. Not positive I love it. Original is commented out.}
  For example, we do not assume a bounded feasible set or a particular
  choice of starting point (as in \cite{GHaeser_HLiu_YYe_2018}), and
  we do not assume any specific rate of growth of $f$ as $x$ moves
  away from the solution set.  Still, our computational complexity
  rates match (for small $n$) or improve on (for large $n$)  those in
  \cite{GHaeser_HLiu_YYe_2018}. Practically speaking,} our algorithm
has the appealing feature that it puts minimal restrictions on the
step size, allowing the line search to take steps that are much closer
to the boundary than the current iterate.


\section{Approximate Optimality Conditions} \label{sec:approxKKT}

We now discuss first- and second-order optimality criteria for
(\ref{eq:fdef}) in a form that can be related to the approximate
optimality criteria \eqref{eq:epsKKT} that are targeted by our
algorithm. We show that points satisfying these necessary conditions
are the limits of sequences of points that satisfy our approximate
criteria \eqref{eq:epsKKT}. We then compare our approximate criteria
with similar conditions that have been proposed previously, and argue
that ours are more appropriate.

\subsection{Deriving Approximate Optimality Conditions from Exact Conditions}

First-order conditions for $x$ to be a solution of \eqref{eq:fdef} are
that there exists a vector $s^* \in \R^n$ such that
\begin{equation} \label{eq:exactKKT}
  \nabla f(x) - s^* = 0, \quad (x,s^*) \ge 0, \quad x_i s^*_i=0 \;\;  \mbox{ for all } i=1,2,\dotsc n.
  \end{equation}
Our second-order condition is a modified version of the condition
derived in \cite{RAndreani_GHaeser_ARamos_PJSSilva_2017}. It requires
the existence of a vector $\theta^*$ such that
\begin{equation} \label{eq:secondorderexact}
  \nabla^2 f(x) + \diag (\theta^*) \succeq 0, \quad
 \theta^* \ge 0, \quad 
x_i^2 \theta_i^* = 0 \; \; \mbox{ for all } i=1,2,\dotsc n.
  \end{equation}
This is equivalent to a ``weak'' form of second-order necessary
conditions for (\ref{eq:fdef}), namely $[\nabla^2
  f(x)]_{\mathcal{I}(x)\mathcal{I}(x)} \succeq 0$, where
$\mathcal{I}(x) := \{i \, | \, x_i > 0 \}$. The more satisfactory
``strong'' second-order conditions require testing that $d^\top \nabla^2
f(x) d \ge 0$ for all $d$ in the cone defined by
\[
\{ d \in \R^n \, | \, d_i=0  \; \mbox{when $x_i=0$, $[\nabla f(x)]_i>0$}; \; d_i \ge 0 \; \mbox{when $x_i=0$, $[\nabla f(x)]_i=0$} \}.
\]
This is known to be an NP-hard problem
\cite{MNouiehed_JDLee_MRazaviyayn_2018}.

The following result shows that a \refer{local minimizer $x^*$} can be expressed
in terms of the limit of sequences that satisfy approximate forms of
these two optimality conditions.
\begin{theorem} \label{thm:cappedKKTlim}
Let $f$ be twice continuously differentiable on the interior of
$\R^n_+$. Let $x^*$ be a local solution of (\ref{eq:fdef}). Then there
exists a sequence of approximate solutions $\{x^k\}$ with $x^k >0$;
sequences of approximate Lagrange multipliers $\{s^k\}$ and
$\{\theta^k\}$, with $s^k \ge 0$ and $\theta^k \ge 0$; and a sequence
of scalars $\{\delta_k\}$ with $\delta_k >0$ and $\delta_k \rightarrow
0$ such that the following conditions hold:
\begin{subequations}
\begin{alignat}{2}
x^k > 0 \mbox{ for all k and } x^k & \rightarrow x^*, && \label{eq:cappedlimx} \\
\nabla f(x^k) - s^k & \rightarrow 0, && \label{eq:cappedlimlagrangian} \\
\min\{x_i^k, 1\} s_i^k & \rightarrow 0 \; \; && \mbox{\rm for all } i=1,2,\dotsc n,
\label{eq:cappedlimcompliment} \\
 \nabla^2 f(x^k) + \diag(\theta^k) + \delta_k I 
 & \succeq 0,
 \label{eq:cappedlim2order} \\
\min \{x_i^k, 1\}^2 \theta_i^k & \rightarrow 0 \; \; && \mbox{\rm for all } i=1,2,\dotsc n.
\label{eq:cappedlimthetacomp}
\end{alignat}
\end{subequations}
\end{theorem}
\refer{The proof of this result follows directly from that
of \cite[Theorem~1]{GHaeser_HLiu_YYe_2018} by noting
that $\min\{x_i^k, 1\} s_i^k \leq x_i^k s_i^k$ and
$\min \{x_i^k, 1\}^2\theta_i^k \leq (x_i^k)^2 \theta_i^k$ trivially hold
for all $i$ and $k$.}

Theorem~\ref{thm:cappedKKTlim} suggests that we should declare $x>0$ to
be an approximate interior solution of (\ref{eq:fdef}) when there
exist $s \in \mathbb{R}^n_+$ and $\theta \in \mathbb{R}^n_+$ such that
\begin{subequations} \label{eq:epscappedKKT}
\begin{align}
\left\|\nabla f(x) - s \right\|_\infty & \leq \epsg, \label{eq:epscappedKKT2} \\
\|\bX s\|_\infty & \leq \epsg, \label{eq:epscappedKKT3} \\
\nabla^2 f(x) + \diag (\theta) + \epsH I
& \succeq 0, \label{eq:epscappedKKT4} \\
\|\bX^2 \theta\|_\infty & \leq \epsH. \label{eq:epscappedKKT5} 
\end{align}
\end{subequations}
We will now describe the connection between our approximate optimality
conditions (\ref{eq:epsKKT}) and the conditions
(\ref{eq:epscappedKKT}).

\begin{theorem}
Let $x$ be a point satisfying (\ref{eq:epsKKT}). Then there exist $s
\in \mathbb{R}^n_+$ and $\theta \in \mathbb{R}^n_+$ such that
(\ref{eq:epscappedKKT}) holds at $x$.
\end{theorem}
\begin{proof}
  Let $s_i := \max\{0, [\nabla
  f(x)]_i\}$ for $i=1\dotsc n$, so that $s \in \R^n_+$ and, by direct
substitution, we have (\ref{eq:epscappedKKT2}) and
(\ref{eq:epscappedKKT3}). Our second-order condition
\eqref{eq:epspsd} is that 
\[
d^\top \left( \bX \nabla^2 f(x)\bX + \epsH I \right)d \geq 0, \quad
\mbox{for all $d \in \R^n$.}
\]
Since $\bX^{-1}$ exists and is positive definite, we have
\[
d^\top \left(\nabla^2 f(x) + \sum_{i=1}^n \frac{\epsH}{\min\{x_i,1\}^2} e_i e_i^\top\right) d
\geq 0, \quad \mbox{for all $d \in \R^n$.} 
\]
Therefore, by choosing $\theta_i = \epsH/\min\{x_i, 1\}^2$ for all
$i=1,2, \dotsc n$, we have that $\theta \ge 0$ and that
(\ref{eq:epscappedKKT4}) and (\ref{eq:epscappedKKT5}) are both
satisfied.
\end{proof}

\subsection{Comparison with Previously Proposed Approximate Conditions}

The conditions (\ref{eq:exactKKT}) and (\ref{eq:secondorderexact})
directly motivate the approximate optimality
conditions for $x >0$ used in the interior-point method of
\cite{GHaeser_HLiu_YYe_2018}, which are
\begin{subequations} \label{eq:oldepsKKT}
\begin{align}
\nabla f(x) &\geq -\epsg e, \label{eq:oldgradpos} \\
\|X \nabla f(x)\|_\infty &\leq \epsg, \label{eq:oldepscompliment} \\
d^\top \left(X \nabla^2 f(x) X + \sqrt{\epsg} I\right) d &\geq 0. \label{eq:old2order}
\end{align}
\end{subequations}
\swmodify{The scaled first-order condition (\ref{eq:oldepscompliment})
  and scaled second-order condition (\ref{eq:old2order}) are commonly
  used optimality conditions for (\ref{eq:fdef})
  \cite{WBian_XChen_YYe_2015, XChen_FXu_YYe_2010}. However, these two
  conditions alone are insufficient to guarantee that a sequence of
  points that satisfies these conditions as $\epsg \to 0$ converges to
  a KKT point for $f$~\cite{GHaeser_HLiu_YYe_2018}. For this reason
  the condition $(\ref{eq:oldgradpos})$ is added
  in~\cite{GHaeser_HLiu_YYe_2018}, motivated by the first-order optimality
  conditions \eqref{eq:exactKKT}.}

These conditions can be overly stringent for coordinates $i$ in which
$x_i \gg 0$.  In this case, the complementarity condition
(\ref{eq:oldepscompliment}), requires $|[\nabla f(x)]_i|$ to be very
small. Similarly, (\ref{eq:old2order}) requires that the Hessian in
the subspace spanned by these coordinates can have only minimal
negative curvature.  Such requirements contrast sharply with the case
of unconstrained minimization.  In the limiting scenario in which all
of the coordinates of $x$ are far from the boundary, these approximate
first-order conditions are significantly harder to satisfy than in the
(equivalent) unconstrained formulation.


To remedy this situation, our approximate optimality conditions
\eqref{eq:epsKKT} contain scalings by $x_i$ only when $x_i \in (0,1]$.
  Our conditions thus interpolate between the bound-constrained case
  (when $x_i$ is small) and the unconstrained case (when $x_i$ is
  large) while also controlling the norm of the matrix used in our
  optimality conditions.



\section{Log-Barrier Newton-CG Algorithm} \label{sec:lbncg}

\begin{algorithm}[ht!]
\caption{Log-Barrier Newton-Conjugate-Gradient}
\label{alg:lbncg}
\begin{algorithmic}
\STATE \emph{Inputs:} Tolerance  $\epsg \in (0, 1)$; backtracking 
parameter $\theta \in (0,1)$; starting point $x^0 > 0$;
accuracy parameters $\zeta_r \in (0,1)$ and $\bar{\zeta} \in (0,1)$;
maximum step \swmodify{scaling} $\beta \in [\epsg^{1/2}, 1)$;
step acceptance parameter $\eta \in (0,1)$; \\
\emph{Optional input:} Scalar  $\hat{M} > 0$ such that $\| \nabla^2 f(x) \| \le \hat{M}$ for all $x$ (set $\hat{M}=0$ if not provided);
\STATE Set $\epsH = \epsg^{1/2}$, $\mu = \epsg/4$, $c_\mu = \bar{\zeta} \mu$,
$M_\mu = \hat{M} + \mu$;
\FOR{$k=0,1,2,\dotsc$}
\IF{$[\nabla f(x^k)]_i \leq -\epsg$ for some coordinate $i$ or $\|\bX_k\nabla f(x^k)\|_\infty > \epsg$}
\STATE Call Algorithm~\ref{alg:ccg} with $H = \bX_k\nabla^2 \phi_\mu (x^k)\bX_k$,
$\epsilon=\epsH$, $g=\bX_k \nabla \phi_\mu(x^k)$, accuracy parameters $\zeta_r$
and $c_\mu$, and  bound $M=M_\mu$,  to obtain outputs $\hat{d}^k$, d\_type;

\IF{\{d\_type=NC\}}
\STATE $d^k \leftarrow -\mathrm{sgn}(g^\top \hat{d}^k) 
\min\left\{\frac{|(\hat{d}^k)^\top \bX_k \nabla^2 \phi_\mu (x^k) \bX_k \hat{d}^k|}
{\|\hat{d}^k\|^3}, \frac{\beta}{\|X_k^{-1} \bX_k \hat{d}^k\|_\infty}\right\} \hat{d}^k$;
\ELSE[d\_type=SOL]
\STATE $d^k \leftarrow \min\left\{1, \frac{\beta}{\|X^{-1}_k \bX_k \hat{d}^k\|_\infty} \right\}
\hat{d}^k$;
\ENDIF
\STATE Go to \textbf{Line Search};
\ELSE
\STATE Call Procedure~\ref{alg:meo} with $H = \bX_k \nabla^2 f (x^k) \bX_k$,
$\epsilon=\epsH$, and $M=\hat{M}$ (if provided); 
\IF{Procedure~\ref{alg:meo} certifies that $\lambdamin(\bX_k \nabla^2 f(x^k) \bX_k)
\ge -\epsH$}
\STATE
Terminate; 
\ELSE[direction of sufficient negative curvature $v$ returned by Procedure~\ref{alg:meo}]
\STATE Set 
$d^k \leftarrow -\mathrm{sgn}(v^\top \bX_k \nabla \phi_\mu(x^k))
\min\left\{|v^\top \bX_k \nabla^2 \phi_\mu (x^k) \bX_k v|,
\frac{\beta}{\|X_k^{-1} \bX_k v\|_\infty}\right\} v$;
\STATE Go to \textbf{Line Search};
\ENDIF
\ENDIF
\STATE \textbf{Line Search}: Compute a step length $\alpha_k=\theta^{j_k}$, 
where $j_k$ is the smallest nonnegative integer such that
\begin{equation} \label{eq:lsdecreasedamped}
	\phi_\mu(x^k + \alpha_k \bX_k d^k) < \phi_\mu(x^k)
	- \frac{\eta}{6}\alpha_k^3 \|d^k\|^{3};
\end{equation}
\STATE $x^{k+1} \leftarrow x^k+\alpha_k \bX_k d^k$;
\ENDFOR
\end{algorithmic}
\end{algorithm}

We now give an overview of our Log-Barrier Newton-CG (LBNCG)
algorithm, defined in Algorithm~\ref{alg:lbncg}, along with its
component parts.

The main branch in each iteration is conditional on the approximate
first-order optimality conditions, (\ref{eq:epsgradpos}) and
(\ref{eq:epscompliment}). When one or both of these conditions are not
satisfied, the Capped CG method (Algorithm \ref{alg:ccg}) is applied
to the damped, preconditioned Newton system
\begin{equation} \label{eq:lbn}
\left(\bX_k \nabla^2
\phi_\mu (x^k) \bX_k + 2\epsH I\right)d = \bX_k \nabla \phi_\mu(x^k),
\end{equation}
where according to the definition \eqref{eq:logbarrierdef} of the
barrier function $\phi_\mu$, we have
\[
\nabla \phi_\mu(x) = \nabla f(x) - \mu X^{-1} e \quad \mbox{ and } \quad
\nabla^2 \phi_\mu(x) = \nabla^2 f(x) + \mu X^{-2}.
\]
Algorithm~\ref{alg:ccg}, which is described further in Section
\ref{subsec:ccg} and in the earlier paper
\cite{CWRoyer_MONeill_SJWright_2019}, returns either an approximate
solution to the linear system \eqref{eq:lbn}, or else a direction of
sufficient negative curvature for $\bX_k \nabla^2 \phi_\mu (x^k)
\bX_k$.

Alternatively, when (\ref{eq:epsgradpos}) and (\ref{eq:epscompliment})
are satisfied, a ``Minimum Eigenvalue Oracle''
(Procedure~\ref{alg:meo}) is invoked to certify either that the
second-order optimality condition (\ref{eq:epspsd}) holds at the
current iterate or, if not, to return a direction $v$ of sufficient
negative curvature for $\bX_k \nabla f(x^k)
\bX_k$. Procedure~\ref{alg:meo} may be implemented by a randomized
procedure, with some probability of failure $\delta$, in which it
incorrectly certifies that (\ref{eq:epspsd}) is satisfied. Further
discussion of this procedure appears in Section~\ref{subsec:meo}.

However the search direction is chosen, it is scaled to obtain a step
$d^k$ that satisfies $\|X_k^{-1} \bX_k d^k\|_\infty \leq \beta < 1$.
This condition guarantees that for $x^k > 0$, we have
\[
x^{k+1} = x^k + \bX_k d^k = X_k \left(e + X_k^{-1} \bX_k d^k\right)
\geq x^k (1 - \beta) > 0,
\]
so that all iterates lie strictly inside the positive orthant. A
backtracking linesearch is performed along the direction $\bX_k d^k$
to ensure sufficient decrease in $\phi_\mu$. We note
that a value of $\beta$ close to its upper bound of $1$ results in
aggressive steps that may approach the zero bounds closely. Steps of
this kind are favored in practical interior-point methods. We will
see in later sections that a factor $(1-\beta)$ emerges in the
complexity results, leading to weaker bounds if $\beta$ is {\em too}
close to $1$. Though we are mindful of this effect, our focus is on
the dependence on the tolerance $\epsg$. The choice of $\beta$ is
independent of $\epsg$; we would not \refer{expect $\beta$ to be updated}
in response to a change in the tolerance $\epsg$.

We set a number of parameters at the beginning of the algorithm,
including the particular choice $\epsH = \epsg^{1/2}$. This choice is
commonly made in the unconstrained optimization literature too, for
purposes of aligning two different complexity expressions. In our
current context, this choice is embedded more deeply into the
analysis, but we keep the distinction between $\epsH$ and $\epsg$ to
maintain the generality of individual results. The particular choice
$\mu = \epsg/4$ of the barrier parameter is key to the complexity
result. Finally, we note that when $\hat{M}$ is an upper bound on $\|
\nabla^2 f(x) \|$ for all $x$ of interest, we have
\begin{equation} \label{eq:hx9}
\|\bX \nabla^2 \phi_\mu(x) \bX\|  \le
\|\bX \nabla^2  f(x) \bX\|  + \mu \| \bX X^{-2} \bX \| \le
\| \nabla^2  f(x) \| + \mu \le \hat{M}+\mu,
\end{equation}
so that $\| H \| \le M_\mu$ for $H$ defined as the input of Algorithm~\ref{alg:ccg} in
Algorithm~\ref{alg:lbncg}.

\subsection{Capped Conjugate Gradient} \label{subsec:ccg}

\begin{algorithm}[ht!]
\caption{Capped Conjugate Gradient}
\label{alg:ccg}
\begin{algorithmic}
\STATE \emph{Inputs:} Symmetric matrix $H \in \R^{n \times n}$; 
vector $g \ne 0$; damping parameter $\epsilon \in (0,1)$; desired relative 
accuracy parameter $\zeta_r \in (0,1)$; desired accuracy $c_\mu \in (0,1)$; 
\STATE \emph{Optional input:} scalar $M \geq 0$ such that
$\| H \| \le M$ (set to $0$ if not provided);
\STATE \emph{Outputs:} d\_type, $d$;
\STATE \emph{Secondary outputs:} final values of $M$, $\kappa$, 
$\hat{\zeta_r}$, $\tau$, and $T$;
\STATE Set
\[
\bH:=H+2\eps I, \quad \kappa := \frac{M+2\eps}{\eps}, \quad \hat{\zeta}_r := \frac{\zeta_r}{3 \kappa},
\quad
\tau := \frac{\sqrt{\kappa}}{\sqrt{\kappa}+1}, \quad
T:=\frac{4\kappa^4}{(1-\sqrt{\tau})^2};
\]
\STATE $y^0 \leftarrow 0$, $r^0 \leftarrow g$, $p^0 \leftarrow -g$, $j \leftarrow 0$; 
\IF{$(p^0)^\top \bH p^0 < \eps \|p^0\|^2$}
\STATE Set $d=p^0$ and terminate with d\_type=NC;
\ELSIF{$\|H p^0\| > M \|p^0\|$}
\STATE Set $M \leftarrow {\|H p^0\|}/{\|p^0\|}$ 
and update $\kappa,\hat{\zeta}_r,\tau,T$ accordingly;
\ENDIF
\WHILE{TRUE}
\STATE $\alpha_j \leftarrow {(r^j)^\top r^j}/{(p^j)^\top \bH p^j}$;
\COMMENT{Begin Standard CG Operations}
\STATE $y^{j+1} \leftarrow y^j+\alpha_j p^j$;
\STATE $r^{j+1} \leftarrow r^j + \alpha_j \bH p^j$;
\STATE $\beta_{j+1} \leftarrow \|r^{j+1}\|^2/\|r^j\|^2$;
\STATE $p^{j+1} \leftarrow  -r^{j+1} + \beta_{j+1}p^j$;
\COMMENT{End Standard CG Operations}
\STATE $j \leftarrow  j+1$;
\IF{$\|H p^j\| > M \|p^j\|$}
\STATE Set $M \leftarrow {\|H p^j\|}/{\|p^j\|}$ 
and update $\kappa,\hat{\zeta}_r,\tau,T$ accordingly;
\ELSIF{$\|H y^j\| > M \|y^j\|$}
\STATE Set $M \leftarrow {\|H y^j\|}/{\|y^j\|}$ 
and update $\kappa,\hat{\zeta}_r,\tau,T$ accordingly;
\ELSIF{$\| H r^j\| > M \|r^j\|$}
\STATE Set $M  \leftarrow {\| H r^j\|}/{\|r^j\|}$ 
and update $\kappa,\hat{\zeta}_r,\tau,T$ accordingly;
\ENDIF
\IF{$(y^j)^\top \bH y^j < \epsilon \|y^j\|^2$}
\STATE Set $d \leftarrow y^j$ and terminate with d\_type=NC;
\ELSIF{$\| r^j \| \le \hat{\zeta}_r \|r^0\|$ and $\|r^j\|_\infty \le c_\mu$}
\STATE Set $d \leftarrow y^j$ and terminate with d\_type=SOL;
\ELSIF{$(p^j)^\top \bH p^j < \eps \|p^j\|^2$}
\STATE Set $d \leftarrow p^j$ and terminate with d\_type=NC;
\ELSIF{$\|r^j\| >  \sqrt{T}\tau^{j/2} \|r^0\|$
}
\STATE Compute $\alpha_j,y^{j+1}$ as in the main loop above;
\STATE Find $i \in \{0,\dotsc,j-1\}$ such that
\begin{equation} \label{eq:weakcurvdir}
	\frac{(y^{j+1}-y^i)^\top \bH (y^{j+1}-y^i)}{\|y^{j+1}-y^i\|^2} \; < \; \eps;
\end{equation}
\STATE Set $d \leftarrow y^{j+1}-y^i$ and terminate with d\_type=NC;
\ENDIF
\ENDWHILE
\end{algorithmic}
\end{algorithm}

Algorithm~\ref{alg:ccg} is a safeguarded version of the conjugate
gradient (CG) procedure for either solving the linear system $(H+2\eps
I) y = - g$, or else detecting a direction $d$ such that $d^\top Hd \le -
\eps \|d\|^2$.  This method, which was described in
\cite{CWRoyer_MONeill_SJWright_2019}, consists of classical CG
iterations plus various checks to determine whether (a) the upper
bound $M$ on $\|H\|$ is adequate, and (b) negative curvature in $H$
has been detected. One of the techniques for detecting negative
curvature is the too-slow-convergence criterion $\|r^j\| >
\sqrt{T}\tau^{j/2} \|r^0\|$ (where $T$ and $\tau$ both depend on the
bound $M$). By Theorem~\ref{thm:cvCGwhilestronglycvx}, this behavior
can occur only when there exists some $i \in \{0,\dotsc,j-1\}$ such
that $(y^{j+1}-y^i)^\top \bH (y^{j+1}-y^i) \; < \;
\eps\|y^{j+1}-y^i\|^2$ holds. In this situation,
Algorithm~\ref{alg:ccg} returns $d = y^{j+1} - y^i$ as a direction of
sufficient negative curvature.


Algorithm~\ref{alg:ccg} is called from Algorithm~\ref{alg:lbncg} with
$H = \bX_k \nabla^2 \phi_\mu (x^k) \bX_k$ which, as we note in
\eqref{eq:hx9}, has norm bounded by $M_\mu = \hat{M}+\mu$, where
$\hat{M}$ is the bound on $\| \nabla^2 f(x^k) \|$. Hence the value of
$M$ in Algorithm~\ref{alg:ccg} will never be larger than this value.

Altogether, the safeguards mentioned above and the diagonal
preconditioning strategy guarantee that Capped CG requires
$\min \{n,\tcO (\epsilon^{-1/2})\}$ iterations to
terminate.
A derivation of this bound is given in Section~\ref{subsec:wccccg}.

\subsection{Minimum Eigenvalue Oracle} \label{subsec:meo}

\floatname{algorithm}{Procedure}

\begin{algorithm}[ht!]
\caption{Minimum Eigenvalue Oracle}
\label{alg:meo}
\begin{algorithmic}
  \STATE \emph{Inputs:} Symmetric matrix $H \in \R^{n \times n}$,
  tolerance $\epsilon>0$;
  \STATE \emph{Optional input:} Scalar $M>0$ such that $\| H \| \le M$;
  \STATE \emph{Outputs:} An estimate $\lambda$ of $\lambdamin(H)$
  such that $\lambda \le - \epsilon/2$, and vector $v$
  with $\|v\|=1$ such that $v^\top H v =\lambda$ OR
  a certificate that $\lambdamin(H) \ge -\epsilon$. In the latter case, when
  the certificate is output, it is false with probability at most $\delta$,
  for some $\delta \in [0,1)$.
\end{algorithmic}
\end{algorithm}

The Minimum Eigenvalue Oracle (Procedure~\ref{alg:meo}) is called when
the approximate first-order conditions \eqref{eq:epsgradpos},
\eqref{eq:epscompliment} are satisfied. This procedure either verifies
that the approximate second-order condition \eqref{eq:epspsd} is
satisfied as well (in which case the algorithm terminates), or else
returns a direction of sufficient negative curvature for the scaled
Hessian $\bX_k \nabla^2 f(x^k) \bX_k$, along which further progress
can be made in reducing the barrier function $\phi_\mu$.


This procedure can be implemented via any method that finds the
smallest eigenvalue of $H$ to an absolute precision of $\epsilon/2$
with probability at least $1-\delta$. (A deterministic implementation
based on a full eigenvalue decomposition would have $\delta=0$.)  In
Section~\ref{subsec:wcc2}, we will establish complexity results under
this general setting, and analyze the impact of the threshold
$\delta$.

Several possibilities for implementing Procedure~\ref{alg:meo} have
been proposed in the literature, with various guarantees.
In our setting, in which Hessian-vector products and vector operations
are the fundamental operations, Procedure~\ref{alg:meo} can be
implemented using the Lanczos method with a random starting vector
(see \cite{YCarmon_JCDuchi_OHinder_ASidford_2018}). The following
result from \cite[Lemma~2]{CWRoyer_MONeill_SJWright_2019} verifies its
effectiveness.



\begin{lemma} \label{lem:randlanczos}
Suppose that the Lanczos method is used to estimate the smallest
eigenvalue of $H$ starting with a random vector uniformly generated on
the unit sphere, where $\|H\| \le M$. For any $\delta \in [0,1)$, this
  approach finds the smallest eigenvalue of $H$ to an absolute
  precision of $\eps/2$, together with a corresponding direction $v$,
  in at most
\begin{equation} \label{eq:rlanc}
	\min \left\{n, 1+\ceil*{\frac12 \ln (2.75 n/\delta^2)
	\sqrt{\frac{M}{\epsilon}}} \right\} \quad \mbox{iterations},
\end{equation}	
with probability at least $1-\delta$.
\end{lemma}

Procedure~\ref{alg:meo} can be implemented by outputting the
approximate eigenvalue $\lambda$ for $H$, determined by the randomized
Lanczos process, along with the corresponding direction $v$, provided
that $\lambda \le -\eps/2$. When $\lambda>-\eps/2$,
Procedure~\ref{alg:meo} returns the certificate that $\lambdamin(H)
\ge -\eps$, \swmodify{a conclusion that is} correct with probability
at least $1-\delta$. Conjugate gradient with a random right-hand side
can be used as an alternative to randomized Lanczos, with essentially
the same properties; see \cite[Appendices~A and
  B]{CWRoyer_MONeill_SJWright_2019}.


\section{Complexity Analysis} \label{sec:wcc}

This section presents complexity results for
Algorithm~\ref{alg:lbncg}.  Section~\ref{subsec:wccccg} describes the
iteration complexity of Capped CG (Algorithm~\ref{alg:ccg}) and the
properties of its outputs. Section~\ref{subsec:wcc1} shows that
Algorithm~\ref{alg:lbncg} {\em deterministically} finds a point
satisfying the approximate first-order optimality conditions
\eqref{eq:epsgradpos}, \eqref{eq:epscompliment} in at most
\refer{$\tcO(n\epsg^{-1/2} + \epsg^{-3/2})$ iterations.
We also show that these conditions are satisfied in
at most $\tcO(n\epsg^{-3/4}+ \epsg^{-7/4})$ gradient evaluations and/or
Hessian-vector products when $n$ is large and
$\tcO(n\epsg^{-3/2})$ operations when $n$ is small.}
Finally, Section~\ref{subsec:wcc2} shows that
the same \swmodify{type of complexity bound holds (differing in the
  constants)} for finding a point which satisfies all approximate
optimality conditions in \eqref{eq:epsKKT} with high probability
(rather than deterministically).

\subsection{Properties of Capped CG} \label{subsec:wccccg}

We begin this subsection by finding a lower bound on the norm of the
right-hand side in the Newton system of Algorithm~\ref{alg:lbncg}
(Lemma~\ref{lem:loggradnorm}). We then derive a bound on the maximum
number of iterations of the Capped CG method that can occur before
returning a direction $\hat{d}^k$, which is either an approximate
solution of \eqref{eq:lbn} or a negative curvature direction for the
diagonally scaled Hessian of the log-barrier function
(Lemma~\ref{lem:ccgits}). Theorem~\ref{thm:cvCGwhilestronglycvx}
verifies that the direction returned in the case of too-slow-decrease
is in fact a vector with the required negative curvature properties.
Finally, we present a number of properties
of the search direction $d^k$ computed from the vector returned by
Algorithm~\ref{alg:ccg}, which will be instrumental in the complexity
analysis of the following sections (Lemma~\ref{lem:ccgsteps}).

\begin{lemma} \label{lem:loggradnorm}
Let $\mu = \epsg/4$ and suppose that either (\ref{eq:epsgradpos})
or (\ref{eq:epscompliment}) is violated at $x^k$. Then,
\begin{equation} \label{eq:loggradnorm}
\|\bX_k \nabla \phi_\mu (x^k)\| \geq \mu.
\end{equation}
\end{lemma}
\begin{proof}
By definition of $\nabla \phi_\mu(x^k)$, we have
\begin{equation} \label{eq:loggradlem1}
\|\bX_k \nabla \phi_\mu (x^k)\| = \|\bX_k \nabla f(x^k) - \mu \bX_k X_k^{-1} e\|.
\end{equation}

Suppose first that (\ref{eq:epsgradpos}) is not satisfied
at $x^k$. Thus, there exists at least one coordinate $i$ such
that $[\nabla f(x^k)]_i < -\epsg < 0$.
If $x_i^k \leq 1$, it follows that
\[
\bar{x}^k_i [\nabla f(x^k)]_i - \frac{\bar{x}^k_i}{x^k_i} \mu
= \bar{x}^k_i [\nabla f(x^k)]_i - \mu < -\mu.
\]
If $x_i^k > 1$, we have $\bar{x}_i^k = 1$ so that
\[
\bar{x}^k_i [\nabla f(x^k)]_i - \frac{\bar{x}^k_i}{x^k_i} \mu
< [\nabla f(x^k)]_i < -\epsg = - 4\mu.
\]
In either case, we have from (\ref{eq:loggradlem1}) that
\[
\|\bX_k \nabla \phi_\mu (x^k)\| \geq \mu.
\]

Now, suppose that (\ref{eq:epscompliment}) does not hold, so that
$|\bar{x}^k_i [\nabla f(x^k)]_i| > \epsg$ for some $i$. Thus, we have
\[
\|\bX_k \nabla \phi_\mu (x^k)\| 
\geq \left|\bar{x}^k_i [\nabla f(x^k)]_i - \mu \frac{\bar{x}^k_i}{x^k_i}\right| 
\geq |\bar{x}^k_i [\nabla f(x^k)]_i| - \mu \frac{\bar{x}^k_i}{x^k_i} 
\geq \epsg - \mu \geq 3\mu,
\]
proving the result.
\end{proof}

We now find the iteration bound on Algorithm~\ref{alg:ccg} that was
foreshadowed in Section~\ref{subsec:ccg}. The precise bound in the
following lemma is based on a quantity $J(M,\eps,\zeta_r,c_\mu)$, for
which the estimate in terms of the accuracy parameter is given following
the lemma.
%
%
\begin{lemma} \label{lem:ccgits}
The number of iterations of Algorithm~\ref{alg:ccg} is bounded by
\[
\min\{n,J(M,\eps,\zeta_r,c_\mu)\},
\]
where $J=J(M,\eps,\zeta_r,c_\mu)$ is the smallest integer such that
\begin{equation} \label{eq:pv8}
\sqrt{T}\tau^{J/2} \| r^0 \| \le \min\left\{\hat{\zeta}_r \|r^0\|,  c_\mu \right\},
\end{equation}
where $M$, $\hat{\zeta}_r$, $T$, and $\tau$ are the values returned by
the algorithm. If all iterates $y_i$ generated by
Algorithm~\ref{alg:ccg} are stored, the number of matrix-vector
multiplications required is bounded by
$\min\{n,J(M,\eps,\zeta_r,c_\mu)\}+1$.
If the iterates $y_i$ must be regenerated in order to define the
direction $d$ returned after \eqref{eq:weakcurvdir}, this bound
becomes $2\min\{n,J(M,\eps,\zeta_r,c_\mu)\}+1$.
\end{lemma}
\begin{proof}
We omit a detailed proof, as the result and proof are identical to
\cite[Lemma~1]{CWRoyer_MONeill_SJWright_2019} modulo a new definition
of $J$. We need only consider the case in which $J < n$, where $J$ is
the index defined in the lemma.  If $\|r^J\| > \sqrt{T} \tau^{J/2}
\|r^0\|$, the last termination test in Algorithm~\ref{alg:ccg} ensures
termination at iteration $J$. In the alternative case $\|r^J\| \leq
\sqrt{T} \tau^{J/2} \|r^0\|$, we have by definition of $J$ that 
\[
\|r^J \| \le \sqrt{T} \tau^{J/2} \|r^0\| \leq \min\left\{\hat{\zeta}_r
\|r^0\|, c_\mu\right\}.
\]
Therefore, $\|r^J\| \leq \hat{\zeta}_r \|r^0\|$ and $\|r^J\|_\infty
\leq \|r^J\| \leq c_\mu$ both hold. Thus, by the termination tests in
Algorithm~\ref{alg:ccg}, termination occurs in this case as well,
completing the proof.
\end{proof}

We can now estimate $J(M,\eps,\zeta_r,c_\mu)$ when
Algorithm~\ref{alg:ccg} is called by Algorithm~\ref{alg:lbncg}
\refer{and Assumption~\ref{assum:boundedgH} holds.}  Here,
we have $r^0 = \bX_k \nabla \phi_\mu(x^k)$ and $c_\mu = \bar{\zeta}
\mu$, so that the right-hand side of condition \eqref{eq:pv8} is
\begin{equation} \label{eq:pv9}
  \min\left\{\hat{\zeta}_r \|\bX_k \nabla \phi_\mu(x^k)\|, \bar{\zeta}
  \mu\right\}.
\end{equation}
Using the same argument as in~\cite{CWRoyer_MONeill_SJWright_2019},
when the minimum in \eqref{eq:pv9} is achieved by the first argument,
we have
\swmodify{
\begin{equation}
\label{eq:HvcountcappedCG1}
J(M,\eps,\zeta_r,c_\mu)  \le
\left\lceil \left(\sqrt{\kappa}+\frac{1}{2}\right)
\ln\left(\frac{144 \left(\sqrt{\kappa}+1\right)^2
  \kappa^6} {\zeta_r^2}\right) \right\rceil
  =  \tcO \left(\eps^{-1/2}\right).
\end{equation}
}
On the other hand, when the minimum in \eqref{eq:pv9} is achieved by
the second argument, an argument
of~\cite{CWRoyer_MONeill_SJWright_2019} along with the bound
\[
\|\bX_k
\nabla \phi_\mu(x^k)\| \leq \|\bX_k \nabla f(x^k)\| + \mu\|\bX_k
X_k^{-1} e\| \leq U_g + \mu \sqrt{n},
\]
shows that
\swmodify{
  \begin{align}
    \nonumber
  J(M,\eps,\zeta_r,c_\mu)  & \le
  \left\lceil \left(\sqrt{\kappa}+\frac{1}{2}\right)
\ln\left(\frac{16 \left(\sqrt{\kappa}+1\right)^2
  \kappa^4 (U_g + \mu \sqrt{n})^2} {\bar{\zeta}^2 \mu^2}\right) \right\rceil 
\\
\label{eq:HvcountcappedCG}
& =  \tcO(\eps^{-1/2}).
\end{align}
}
Therefore, in either case, we have that \swmodify{$J(M,\eps,\zeta_r,
 c_\mu) \leq \tcO(\eps^{-1/2})$}, as claimed in
Section~\ref{subsec:ccg}.

The following theorem shows that when Algorithm~\ref{alg:ccg} is
terminated because of the test $\|r^j\| > \sqrt{T}\tau^{j/2} \|r^0\|$,
then (\ref{eq:weakcurvdir}) will hold for some $i = 0,1, \dotsc, j$,
so that the outputs of Algorithm~\ref{alg:ccg} are well defined.
\begin{theorem} \label{thm:cvCGwhilestronglycvx}
Suppose that the main loop of Algorithm~\ref{alg:ccg} terminates with
$j=\hIcg$, where
\[
\hIcg \in \{1,\dotsc,\min\{n,J(M,\eps,\zeta_r,c_\mu)\}\},
\]
(where $J(M,\eps,\zeta_r,c_\mu)$ is defined in Lemma~\ref{lem:ccgits})
because the fourth termination test is satisfied and the three earlier
conditions do not hold, that is, $(y^{\hIcg})^\top \bH y^{\hIcg} \ge
\eps \|y^{\hIcg}\|^2$, $(p^{\hIcg})^\top \bH p^{\hIcg} \ge \eps
\|p^{\hIcg}\|^2$,
\[
\|r^{\hIcg}\| > \hat{\zeta}_r\|r^0\| \quad \mbox{and/or} \quad
\|r^{\hIcg}\|_\infty > c_\mu,
\]
and
\begin{equation} \label{eq:Kitsstronglycvx}
\|r^{\hIcg}\| > \sqrt{T}\tau^{\hIcg /2} \|r^0\|
\end{equation}
where $M$, $T$, and $\tau$ are the values returned by
Algorithm~\ref{alg:ccg}. Then $y^{\hIcg+1}$ is computed by
Algorithm~\ref{alg:ccg}, and we have
\begin{equation} \label{eq:yKp1negcurv}
\frac{(y^{\hIcg +1}-y^i)^\top \bH (y^{\hIcg+1}-y^i)}
     {\|y^{\hIcg+1}-y^i\|^2} < \eps, \quad \mbox{for some $i \in
       \{0,\dotsc,\hIcg-1\}$.}
\end{equation}
\end{theorem}
\begin{proof}
This result follows directly from
\cite[Theorem~2]{CWRoyer_MONeill_SJWright_2019} after noting that the
properties of $\hIcg$ used in the proof do not depend on the
definition of $J(M,\eps,\zeta_r,c_\mu)$. In particular, $\hIcg$ simply
needs to be an index such that (\ref{eq:Kitsstronglycvx}) holds and
the CG process has not stopped iterating before reaching
$\hIcg$. Thus, the result holds once we account for the additional
stopping criterion $\|r^{\hIcg}\|_\infty \leq c_\mu$ in the new
definition of $J(M,\eps,\zeta_r,c_\mu)$.
\end{proof}

We focus now on the main output of Algorithm~\ref{alg:ccg}, which is
denoted by $\hat{d}^k$ in Algorithm~\ref{alg:lbncg}.  The properties
of $d^k$, which is obtained by scaling $\hat{d}^k$, are essential to
the first- and second-order complexity analysis of later sections.
\begin{lemma} \label{lem:ccgsteps}
Let \refer{Assumption~\ref{assum:fC22}} hold
and suppose that Algorithm~\ref{alg:ccg} is invoked at iteration $k$
of Algorithm~\ref{alg:lbncg}. Let $d^k$ be the vector obtained in
Algorithm~\ref{alg:lbncg} from the output $\hat{d}^k$ of
Algorithm~\ref{alg:ccg}. For each of the two possible settings of
output flag $d\_type$, we have the following.
\begin{enumerate}
\item When {\em d\_type=SOL},  the direction $d^k$ satisfies 
\begin{subequations} \label{eq:ccgstepSOLconds}
\begin{align}
\label{eq:ccgstepSOLcurv}
\epsH \|d^k\|^2 &\leq
(d^k)^\top \left(\bX_k \nabla^2 \phi_\mu (x^k) \bX_k + 2\epsH I\right) d^k, \\
\label{eq:ccgstepSOLnorm}
\|d^k\| &\leq 1.1 \epsH^{-1} \|\bX_k \nabla \phi_\mu(x^k)\|, \\
\label{eq:ccgstepSOLgradcond}
(d^k)^\top \bX_k \nabla \phi_\mu(x^k) &=
-\gamma_k (d^k)^\top \left(\bX_k \nabla^2 \phi_\mu (x^k) \bX_k + 2\epsH I\right) d^k,
\end{align}
\end{subequations}
where $\gamma_k  = \max\left\{\frac{\|X_k^{-1} \bX_k \hat{d}^k\|_\infty}{\beta}, 1\right\}$.
If $\|X_k^{-1} \bX_k \hat{d}^k\|_\infty \leq \beta$ holds, then $d^k$ also satisfies
\begin{equation}
\label{eq:ccgstepSOLres}
\|\hat{r}^k\|  \leq \frac{1}{2}\epsH\zeta_r \| d^k \|,
\end{equation}
where $\hat{r}^k$ is the residual of the scaled Newton system, defined
by
\begin{equation} \label{eq:def.rhat}
\hat{r}^k := \left(\bX_k \nabla^2 \phi_\mu (x^k) \bX_k +2\epsH I\right) \hat{d}^k
+ \bX_k \nabla \phi_\mu(x^k).
\end{equation}
\item When {\em d\_type=NC}, the direction $d^k$ satisfies $(d^k)^\top
  \bX_k \nabla \phi_\mu (x^k) \le 0$ and
  \begin{equation} \label{eq:ccgstepNC}
  \frac{(d^k)^\top \bX_k \nabla^2 \phi_\mu(x^k) \bX_k d^k}{\|d^k\|^2} \leq -\|d^k\| \leq - \epsH.
  \end{equation}
\end{enumerate}
\end{lemma}
\begin{proof}
  For simplicity of notation, we use the following shorthand in the proof:
  \[
  H = \bX_k \nabla^2 \phi_\mu (x^k) \bX_k, \quad
  g = \bX_k \nabla \phi_\mu(x^k).
  \]
  Since Algorithm \ref{alg:lbncg} invoked Algorithm
  \ref{alg:ccg}, at least one of the conditions (\ref{eq:epsgradpos})
  or (\ref{eq:epscompliment}) must be violated at $x^k$. Thus, by
  Lemma~\ref{lem:loggradnorm}, we have $\| g \| \geq \mu > 0$, so the
  iterates of Algorithm~\ref{alg:ccg} are well defined.

  Consider first the case of d\_type=SOL. The bounds
  (\ref{eq:ccgstepSOLcurv}) and (\ref{eq:ccgstepSOLnorm}) follow by
  the same argument as in the first part of the proof of
  \cite[Lemma~3]{CWRoyer_MONeill_SJWright_2019}.
We now prove (\ref{eq:ccgstepSOLgradcond}).  The residual $\hat{r}^k$
at the final iteration of CG procedure is orthogonal to all previous
search directions, so that $(\hat{d}^k)^\top \hat{r}^k = 0$ (see
\cite[Appendix~A]{CWRoyer_MONeill_SJWright_2019}).  Since $\hat{d}^k$
and $d^k$ are collinear, we have $(d^k)^\top \hat{r}^k = 0$, so from
(\ref{eq:def.rhat}) it follows that
\begin{equation}
(d^k)^\top g = -(d^k)^\top (H +2\epsH I) \hat{d}^k.
\end{equation}
When $\|X_k^{-1} \bX_k \hat{d}^k\|_\infty \leq \beta$, we have $d^k =
\hat{d}^k$, so
\[
(d^k)^\top g = -(d^k)^\top (H +2\epsH I) \hat{d}^k = -(d^k)^\top (H
+2\epsH I) d^k,
\]
proving (\ref{eq:ccgstepSOLgradcond}) in this case. When $\|X_k^{-1}
\bX_k \hat{d}^k\|_\infty > \beta$, we have 
\[
d^k =
\frac{\beta}{\|X_k^{-1} \bX_k \hat{d}^k\|_\infty} \hat{d}^k
\]
and thus
\[
(d^k)^\top g = -(d^k)^\top (H +2\epsH I) \hat{d}^k = -
\frac{\|X_k^{-1} \bX_k \hat{d}^k\|_\infty}{\beta} (d^k)^\top (H
+2\epsH I) d^k,
\]
proving (\ref{eq:ccgstepSOLgradcond}) for this case as well.

Turning to (\ref{eq:ccgstepSOLres}), we note first that from
termination conditions of Algorithm~\ref{alg:ccg} that $\|\hat{r}^k\|
\leq \hat{\zeta}_r \|g\|$.  Thus, using (\ref{eq:def.rhat}), we have
that
\[
\|\hat{r}^k\| \leq \hat{\zeta}_r \|g\| \leq 
\hat{\zeta}_r \left(\|(H + 2\epsH I) \hat{d}^k\| + \|\hat{r}^k\|\right)
\leq \hat{\zeta}_r \left((M + 2\epsH) \|\hat{d}^k\| + \|\hat{r}^k\|\right),
\]
where $M$ is the value that is returned by Algorithm~\ref{alg:ccg}, so
that
\[
\|\hat{r}^k\| \leq \frac{\hat{\zeta}_r}{1 - \hat{\zeta}_r} (M + 2\epsH) \|\hat{d}^k\|.
\]
Using again that $\hat{\zeta}_r = \zeta_r/(3 \kappa) < 1/6$ and the
definition of $\hat{\zeta}_r$ in Algorithm~\ref{alg:ccg}, we have
\[
\frac{\hat{\zeta}_r}{1 - \hat{\zeta}_r} (M + 2\epsH)
\leq \frac65 \hat{\zeta}_r (M + 2\epsH)
= \frac65 \frac{\zeta_r \epsH}{3} < \frac12 \zeta_r \epsH,
\]
which yields (\ref{eq:ccgstepSOLres}) when we note that $d^k =
\hat{d}^k$ when $\|X_k^{-1} \bX_k \hat{d}^k\|_\infty \leq \beta$.

In the case of d\_type=NC, we recall that Algorithm~\ref{alg:lbncg}
defines
\begin{equation} \label{eq:sd8}
d^k = -\mathrm{sgn}(g^\top \hat{d}^k) 
\min\left\{\frac{|(\hat{d}^k)^\top H \hat{d}^k|}
{\|\hat{d}^k\|^3}, \frac{\beta}{\|X_k^{-1} \bX_k \hat{d}^k\|_\infty}\right\} \hat{d}^k.
\end{equation}
We have from positivity of the ratios in the $\min \{ \cdot, \cdot\}$
expression that
\[
\mathrm{sgn} (g^\top d_k) = - \mathrm{sgn}(g^\top \hat{d}^k)^2 = -1,
\]
so that $g^\top d_k \le 0$. Next, since $\hat{d}^k$ and $d^k$ are
collinear, we have
\[
\frac{(d^k)^\top (H + 2\epsH I) (d^k)}{\|d^k\|^2} = 
\frac{(\hat{d}^k)^\top (H + 2\epsH I) (\hat{d}^k)}{\|\hat{d}^k\|^2} \leq \epsH,
\]
so that
\begin{equation} \label{eq:ccgsteplemeq3}
\frac{(d^k)^\top H (d^k)}{\|d^k\|^2} \leq -\epsH.
\end{equation}
When the $\min$ in \eqref{eq:sd8} is achieved by the first term, we
have
\[
\|d^k\| = \frac{|(\hat{d}^k)^\top H \hat{d}^k|}{\|\hat{d}^k\|^2} \ge \epsH,
\]
proving (\ref{eq:ccgstepNC}) in this case.  Otherwise, when the $\min$
in \eqref{eq:sd8} is achieved by the second term, we have
\[
\beta =
\|X_k^{-1} \bX_k d^k\|_\infty \leq \|X_k^{-1} \bX_k d^k\| \leq
\|X_k^{-1} \bX_k\| \|d^k\| \leq \|d^k\|.
\]
Using this bound, along with
(\ref{eq:ccgsteplemeq3}) and the fact that $\beta \geq\epsH$ (by
definition), we have
\[
\|d^k\| \geq \min \left\{\frac{|(\hat{d}^k)^\top H \hat{d}^k|}{\|\hat{d}^k\|^2}, \beta \right\}
=\min \left\{\frac{|(d^k)^\top H (d^k)|}{\|d^k\|^2}, \beta \right\}
\geq \min\{\epsH, \beta\} = \epsH.
\]
In either case of the $\min$ in \eqref{eq:sd8}, we have $\|d^k\| \leq
-(d^k)^\top H d^k/\|d^k\|^2$, so that 
\[
\frac{(d^k)^\top H d^k}{\|d^k\|^2} \leq - \|d^k\| \leq -\epsH,
\]
proving (\ref{eq:ccgstepNC}).
\end{proof}

\subsection{First-Order Complexity Analysis} \label{subsec:wcc1}

We now derive a worst-case complexity result for the first-order
optimality condtions (\ref{eq:epsgradpos})
and (\ref{eq:epscompliment}). We show that when
Algorithm~\ref{alg:ccg} returns d\_type=SOL and a unit step is taken
by the line search procedure in Algorithm~\ref{alg:lbncg} (that is,
$\alpha_k = 1$), either the first-order optimality conditions hold at
$x^{k+1}$, or else $\|d^k\|$ is large enough to make significant
progress in reducing the function $\phi_\mu$.
Theorem~\ref{thm:wcc1} and Corollary~\ref{coro:wcc1Hv} state
first-order complexity results in terms of the number of iterations of
Algorithm~\ref{alg:lbncg} and the number of gradient evaluations
and/or Hessian vector products, respectively.

Our results depend on the following technical result concerning the
decrease of the log-barrier term in $\phi_\mu$. Its proof can be found
in Appendix~\ref{app:proof.logbound}.
\begin{lemma} \label{lem:logbound}
  Given $x > 0$, define $X$, $\bX$ as in \eqref{eq:Xdef}, and suppose
  that $d \in \R^n$ is such that $\|X^{-1} \bX d\|_\infty \leq \beta <
  1$. Then,
\begin{multline} \label{eq:logbound}
- \sum_{i=1}^n \log(x_i + \bar{x}_i d_i) + \sum_{i=1}^n \log(x_i) \\
\leq -e^\top X^{-1} \bX d + \frac{1}{2} d^\top \bX X^{-2} \bX d
+ \frac{2 - \beta}{6(1-\beta)^2}\|d\|^3 .
\end{multline}
\end{lemma}

Our first result deals with the case in which a full step
($\alpha_k=1)$ is taken in Algorithm~\ref{alg:lbncg}.


\begin{lemma} \label{lem:shortdsmallKKT}
Let \refer{Assumption~\ref{assum:fC22}} hold
and suppose that Algorithm~\ref{alg:ccg} is invoked at an iterate
$x^k$ of Algorithm~\ref{alg:lbncg}, and returns d\_type = SOL.  Then,
when the unit step is taken (that is, $x^{k+1} = x^k + \bX_k d^k$), we
have either
\begin{equation} \label{eq:dnorm}
  \|d^k\| \geq c_d \epsH, \quad \mbox{where} \;\;
c_d = \min \left\{\frac{1-\bar{\zeta}}{9},
\left(\frac{3}{2 L_H}\right)^{1/2},
\frac{1}{2\left(L_H + 9/2 + \zeta_r \right)}\right\},
\end{equation}
or else 
\begin{equation} \label{eq:ak9}
\nabla f(x^{k+1}) \geq -\epsg e \quad \mbox{and} \quad
\|\bX_{k+1} \nabla f(x^{k+1})\|_\infty \leq \epsg.
\end{equation}
\end{lemma}
\begin{proof}
We begin by noting that if the output $\hat{d}^k$ from
Algorithm~\ref{alg:ccg} satisfies $\|X_k^{-1} \bX_k \hat{d}^k\|_\infty
\geq \beta$ then
\[
\epsH \leq \beta = \|X_k^{-1} \bX_k d^k\|_\infty \leq \|X_k^{-1} \bX_k
d^k\| \leq \|X_k^{-1} \bX_k\| \|d^k\| \leq \|d^k\|,
\]
so the claim (\ref{eq:dnorm}) holds, since $c_d \leq 1$. Thus, we
assume for the remainder of the proof that $\|X_k^{-1} \bX_k
\hat{d}^k\|_\infty < \beta$ and $d^k = \hat{d}^k$, and that $\|d^k\| <
c_d \epsH$. We show that the conditions \eqref{eq:ak9} hold in this
case.

We start by establishing that $\nabla f(x^{k+1}) \geq -\epsg e$.
Since d\_type = SOL, we have that $\bar{\zeta} \mu \geq
\|\hat{r}^k\|_\infty$ where $\hat{r}^k$ is defined in
(\ref{eq:def.rhat}). Using $\|\bX_k X_k^{-2} \bX_k\| \leq 1$ and
$\epsH \|d^k\| < c_d \epsH^2 = c_d \epsg$, it follows that
\begin{align}
\bar{\zeta} \mu &\geq
\|\left(\bX_k \nabla^2 \phi_\mu(x^k)\bX_k + 2 \epsH I\right)d^k
+ \bX_k \nabla \phi_\mu(x^k)\|_\infty \nonumber \\
&= \|\bX_k \left(\nabla^2 f(x^k) \bX_k d^k + \nabla \phi_\mu(x^k) \right)
+ \mu \bX_k X_k^{-2} \bX_k d^k + 2 \epsH d^k\|_\infty \nonumber \\
&\geq \|\bX_k \left(\nabla^2 f(x^k) \bX_k d^k + \nabla \phi_\mu(x^k) \right)\|_\infty
- \mu \|\bX_k X_k^{-2} \bX_k d^k\|_\infty - 2 \epsH \|d^k\|_\infty \nonumber \\
&\geq \|\bX_k \left(\nabla^2 f(x^k) \bX_k d^k + \nabla \phi_\mu(x^k) \right)\|_\infty
- \mu \|\bX_k X_k^{-2} \bX_k\| \|d^k\| - 2 \epsH \|d^k\| \nonumber \\
&\geq \|\bX_k \left(\nabla^2 f(x^k) \bX_k d^k + \nabla \phi_\mu(x^k) \right)\|_\infty
- \mu \|d^k\| - 2 \epsH \|d^k\| \nonumber \\
&> \|\bX_k \left(\nabla^2 f(x^k) \bX_k d^k + \nabla f(x^k) \right) - \mu \bX_k X_k^{-1} e \|_\infty
- c_d \epsH \mu  - 2 c_d \epsg. \label{eq:lemshortdeq1}
\end{align}
Since $\epsH < 1$ and $\mu = \epsg/4$,
we have 
\[
\bar{\zeta} \mu + c_d \epsH \mu + 2 c_d \epsg
\leq \bar{\zeta} \mu + c_d \mu + 2 c_d \epsg
= \mu \left( \bar{\zeta} + 9 c_d \right).
\]
Then, by the definition of $c_d$, $\bar{\zeta} + 9 c_d \leq 1$ so that
$\bar{\zeta} \mu + c_d \epsH \mu + 2 c_d \epsg \leq \mu$. Thus,
by substituting into (\ref{eq:lemshortdeq1}), we obtain
\begin{equation} \label{eq:lemshortdeq2}
\mu > \|\bX_k \left(\nabla^2 f(x^k) \bX_k d^k + \nabla f(x^k) \right) - \mu \bX_k X_k^{-1} e\|_\infty.
\end{equation}
By considering each component $i=1,2,\dotsc, n$ in turn, we now show
that
\begin{equation} \label{eq:lemshortdeq3}
\nabla^2 f(x^k) \bX_k d^k + \nabla f(x^k) > - \mu e.
\end{equation}
When $0 < x^k_i \leq 1$, it follows that $\bar{x}^k_i/x^k_i = 1$, so
\[
\left|\left[\bX_k \left(\nabla^2 f(x^k) \bX_k d^k + \nabla f(x^k) \right)\right]_i - \mu\right| < \mu,
\]
so that
\[
\left[\bX_k \left(\nabla^2 f(x^k) \bX_k d^k + \nabla f(x^k) \right)\right]_i > 0,
\]
establishing (\ref{eq:lemshortdeq3}) for this component $i$.  When
$x^k_i > 1$, we have $\bar{x}^k_i=1$ and $0 < \bar{x}^k_i/x^k_i < 1$,
so from (\ref{eq:lemshortdeq2}), we have
\begin{align*}
-\mu < \left[\bX_k \left(\nabla^2 f(x^k) \bX_k d^k + \nabla f(x^k) \right)\right]_i
- \frac{\bar{x}^k_i}{x^k_i} \mu
& < \left[\bX_k \left(\nabla^2 f(x^k) \bX_k d^k + \nabla f(x^k) \right)\right]_i \\
&= \left[\nabla^2 f(x^k) \bX_k d^k + \nabla f(x^k) \right]_i,
\end{align*}
establishing (\ref{eq:lemshortdeq3}) for this component too.

Finally, using (\ref{eq:fliphessian}), $\mu = \epsg/4$, $\|d^k\| < c_d
\epsH$, $c_d \leq \sqrt{3/(2L_H)}$, and $\epsH^2 = \epsg$, together
with $\|\bX_k\| \leq 1$, we have from (\ref{eq:lemshortdeq3}) that
\begin{align*}
\nabla f(x^{k+1}) &= \nabla f(x^{k+1}) - \nabla^2 f(x^k) \bX_k d^k - \nabla f(x^k)
+ \nabla^2 f(x^k) \bX_k d^k + \nabla f(x^k) \\
&> - \|\nabla f(x^{k+1}) - \nabla^2 f(x^k) \bX_k d^k - \nabla f(x^k)\| e - \mu e \\
&\geq -\frac{L_H}{2} \|\bX_k\|^2 \|d^k\|^2 e -\mu e \\
&> -\left(\frac{L_H}{2} c_d^2 + \frac14\right) \epsg e \geq -\epsg e.
\end{align*}

We now focus on the second condition, $\|\bX_{k+1} \nabla f(x^{k+1})\|
\leq \epsg$.  To begin, we show that
\begin{equation} \label{eq:bk1tobk}
\|\bX_{k+1} \nabla f(x^{k+1})\|_\infty \leq 2\|\bX_k \nabla f(x^{k+1})\|_\infty.
\end{equation}
First, assume that $x^k_i \leq 1$ holds. Then, $\bar{x}_i^k = x_i^k$ so that
$d_i^k = \left(\bar{x}^k_i/x^k_i\right) d_i^k \leq \beta < 1$, so
\[
\bar{x}^{k+1}_i \leq x^{k+1}_i = x_i^k + \bar{x}^k_i d^k_i =  \bar{x}^k_i (1 + d^k_i) < 2\bar{x}^k_i.
\]
When $x^k_i > 1$, we have
\[
\bar{x}^{k+1}_i \leq 1 = \bar{x}^k_i < 2\bar{x}^k_i.
\]
Applying these two cases for each coordinate $i$, we obtain
(\ref{eq:bk1tobk}). Now, recall from the conditions stated at the
start of the proof that $\|X_k^{-1} \bX_k \hat{d}^k\|_\infty < \beta$,
so that $d^k = \hat{d}^k$, where $\hat{d}^k$ is the output of
Algorithm \ref{alg:ccg} at iteration $k$. We thus have for $\hat{r}^k$
defined by \eqref{eq:def.rhat} that
 (\ref{eq:ccgstepSOLres}) holds, by
Lemma \ref{lem:ccgsteps}. 
Therefore, by (\ref{eq:fliphessian}), (\ref{eq:ccgstepSOLres}), (\ref{eq:bk1tobk}),
$\|\bX_k X_k^{-1} e\|_\infty \leq 1$, and $\|\bX_k\| \leq 1$, we have
\begin{align*}
& \|\bX_{k+1} \nabla f(x^{k+1})\|_\infty \\
&\leq 2\| \bX_k \nabla f(x^{k+1})\|_\infty && \text{by (\ref{eq:bk1tobk})} \\
&= 2\|\bX_k \nabla f(x^{k+1}) - \bX_k\nabla f(x^k) + \bX_k\nabla f(x^k)\|_\infty \\
&= 2\|\bX_k \nabla f(x^{k+1}) - \bX_k\nabla f(x^k) - \bX_k \nabla^2 \phi_\mu (x^k) \bX_k d^k\\
&\quad- 2\epsH d^k + \mu \bX_k X_k^{-1} e + \hat{r}^k\|_\infty && \text{by (\ref{eq:def.rhat})}\\
&\leq 2\|\bX_k\left(\nabla f(x^{k+1}) - \nabla f(x^k) - \nabla^2 f(x^k) \bX_k d^k\right)\|_\infty \\
&\quad+ 2 \mu \|\bX_k X_k^{-2} \bX_k d^k\|_\infty + 4 \epsH \|d^k\|_\infty \\
&\quad+ 2\mu \|\bX_k X_k^{-1} e\|_\infty + 2\|\hat{r}^k\|_\infty && \text{by definition of $\phi_\mu$}\\
&\leq 2\|\bX_k\| \|\nabla f(x^{k+1}) - \nabla f(x^k) - \nabla^2 f(x^k) \bX_k d^k\| \\
&\quad+ 2 \mu \|\bX_k X_k^{-2} \bX_k d^k\| + 4 \epsH \|d^k\|
+ 2\mu + 2\|\hat{r}^k\|  && \text{since $\|\bX_k X_k^{-1} e\|_\infty \leq 1$}\\
&\leq L_H \|\bX_k d^k\|^2 + 2 \mu \|\bX_k X_k^{-2} \bX_k\| \|d^k\| \\
&\quad+ 4 \epsH \|d^k\| + 2\mu + \zeta_r \epsH \|d^k\|
&& \text{by (\ref{eq:fliphessian}), (\ref{eq:ccgstepSOLres}), and $\|\bX_k\| \leq 1$} \\
&< L_H c_d^2 \epsg + 2 \mu c_d \epsH + 4 c_d \epsg +
\epsg/2 + \zeta_r c_d \epsg,
\end{align*}
where we used $\|\bX_k\| \leq 1$, $\| \bX_k^{-1} X_k \| \le 1$,
$\|d^k\| < c_d \epsH$, $\epsH^2 = \epsg$,
and $\mu = \epsg/4$ for the last inequality.
Finally, since $\epsH < 1$, $c_d \leq 1$, and
$c_d \leq 1/\left(2\left(L_H + 9/2 + \zeta_r \right) \right)$,
it follows that
\begin{align*}
\|\bX_{k+1} \nabla f(x^{k+1})\|_\infty &< L_H c_d^2 \epsg + 2 \mu c_d \epsH + 4 c_d \epsg
+ \epsg/2 + \zeta_r c_d \epsg\\
&\leq L_H c_d \epsg + 2 \mu c_d + 4 c_d \epsg + \epsg/2 + \zeta_r c_d \epsg\\
&\leq L_H c_d \epsg + c_d \epsg/2 + 4 c_d \epsg + \epsg/2 + \zeta_r c_d \epsg \\
&\leq c_d \epsg \left( L_H + 9/2 + \zeta_r \right) + \epsg/2 \\
&\leq \epsg/2 + \epsg/2 = \epsg,
\end{align*}
completing the proof.
\end{proof}

Lemma~\ref{lem:shortdsmallKKT} is useful in the following line search
argument, because we need only consider cases in which $\|d^k\| \geq
c_d \epsH$. We now show that a sufficiently long step
is taken whenever d\_type=SOL and $x^{k+1}$ does not satisfy the
approximate first-order conditions (\ref{eq:epsgradpos}) and
(\ref{eq:epscompliment}).
\begin{lemma} \label{lem:LSccgSOL}
Suppose that \refer{Assumption~\ref{assum:fC22}} holds.
Suppose that at iteration $k$ of Algorithm~\ref{alg:lbncg}, we have either
$[\nabla f(x^k)]_i \leq -\epsg$ for some coordinate $i$ or
$\|\bX_k \nabla f(x^k)\|_\infty \geq \epsg$, so that Algorithm~\ref{alg:ccg} is called.
When Algorithm~\ref{alg:ccg} outputs a direction $\hat{d}^k$ with d\_type=SOL,
then either
\begin{enumerate}[label=(\Alph*)]
\item the backtracking line search terminates with $\alpha_k = 1$ and
both (\ref{eq:epsgradpos}) and (\ref{eq:epscompliment}) hold at $x^{k+1}$, or
\item the backtracking line search requires at most $j_k \le \jsol+1$ iterations,
where
\begin{equation} \label{eq:ccglsitsSOL}
\jsol \; = \; \left[ \frac12 \log_{\theta}\left( \frac{6(1-\beta)^2}
{(L_H + \eta)(1-\beta)^2 + (2 - \beta)} \; \frac{\epsH^2}{1.1(U_g + \mu \sqrt{n})} \right)
\right]_+,
\end{equation}
and 
\begin{equation}
\refer{\alpha_k \|d^k\| \geq \csol \epsH,}
\end{equation}
where 
\begin{equation*}
\csol = \min\left\{c_d, \frac{6(1-\beta)^2\theta^2}
{(L_H + \eta)(1-\beta)^2 + (2 - \beta)}\right\},
\end{equation*}
and $c_d$ is defined in (\ref{eq:dnorm}).
\end{enumerate}
\end{lemma}
\begin{proof}
This result follows by largely the same argument as that of the proof of \cite[Lemma~13]{CWRoyer_SJWright_2018}.
The main difference is due to the result of Lemma \ref{lem:logbound} which,
together with (\ref{eq:f2smooth}), implies
\begin{equation} \label{eq:phitaylorexpansion}
\phi_\mu(x^k + \theta^j \bX_k d^k) - \phi_\mu(x^k) \leq \theta^j g^\top d^k
+ \frac{\theta^{2j}}{2} (d^k)^\top H d^k + \frac{L_H(1-\beta)^2 + (2-\beta)}{6(1-\beta)^2} \theta^{3j} \|d^k\|^3,
\end{equation}
where the notation $g = \bX_k \nabla \phi_\mu(x^k)$ and $H = \bX_k \nabla^2 \phi_\mu(x^k) \bX_k$ is used once more.
Replacing the Taylor series expansion around $f$ in the proof of \cite[Lemma~13]{CWRoyer_SJWright_2018}
with this expression yields the result. We provide a full proof in Appendix \ref{app:proof.LSccgSOL}.
\end{proof}

Now we show that a sufficiently long step always occurs when
d\_type=NC.
\begin{lemma} \label{lem:LSccgNC}
Suppose that \refer{Assumption~\ref{assum:fC22}} holds.
Suppose that at iteration $k$ of Algorithm~\ref{alg:lbncg}, we have either
$[\nabla f(x^k)]_i \leq -\epsg$ for some coordinate $i$ or\\
$\|\bX_k \nabla f(x^k)\|_\infty \geq \epsg$, so that Algorithm~\ref{alg:ccg} is called.
When Algorithm~\ref{alg:ccg} outputs a direction $\hat{d}^k$ with d\_type=NC,
then the backtracking line search requires at most $j_k \le \jnc+1$ iterations,
where
\begin{equation} \label{eq:ccglsitsNC}
\jnc \; = \; \left[ \log_{\theta}\left(
  \frac{3(1-\beta)^2}{(L_H + \eta) (1-\beta)^2 + (2 - \beta)} \right) \right]_+,
\end{equation}
and 
\begin{equation}
\refer{\alpha_k \|d^k\| \geq \cnc \epsH,}
\end{equation}
where 
\begin{equation*}
\cnc = \min\left\{1, \frac{3(1-\beta)^2 \theta}
{(L_H + \eta) (1-\beta)^2 + (2 - \beta)} \right\}.
\end{equation*}
\end{lemma}
\begin{proof}
This result follows from the same argument as the proof of \cite[Lemma~1]{CWRoyer_SJWright_2018}.
The main difference in the proof once again revolves around the use of (\ref{eq:phitaylorexpansion})
in place of the Taylor expansion around $f$. A full proof is provided in Appendix \ref{app:proof.LSccgNC}.
\end{proof}

\swmodify{Next, we bound the maximum decrease in the logarithmic terms
  over the iterations of Algorithm~\ref{alg:lbncg}.}
\begin{lemma} \label{lem:logdecrease}
\swmodify{Let $\omega$ be such that $\| x^0 \|_{\infty} \le
  \omega$. Then for any $k \ge 0$, we have
\begin{equation} \label{eq:logdecrease}
\sum_{i=1}^n \left( -\log x^{k+1}_i + \log x^0_i \right)
\geq - n\left(\log \omega -\min_i \log x^0_i \right)
-\frac{\sqrt{n}}{\omega} \sum_{j=0}^k \alpha_j \|d^j\|.
\end{equation}
}
\end{lemma}
\begin{proof}
  \swmodify{We focus on a single coordinate $i$, and show that the
    following holds for any $k \ge 0$:
    \begin{equation} \label{eq:qk1}
       -\log x^{k+1}_i + \log x^0_i \ge - \log \omega + \log
       x^0_i  - \frac{1}{\omega} \sum_{j=0}^k \alpha_j |d^j_i|.
    \end{equation}
    We consider three cases.
    \begin{itemize}
      \item[1:] $x^{k+1}_i \le \omega$. Here we have $-\log x^{k+1}_i
        \geq - \log \omega $, so \eqref{eq:qk1} is satisfied
        trivially.
      \item[2:] $x^{k+1}_i > \omega$ and $x^k_i \le \omega$. Here, we have
\begin{align*}
-\log x_i^{k+1}  = -\log\left(x_i^k + \alpha_k \bar{x}_i^k d_i^k\right)
&\geq -\log\left(\omega + \alpha_k \bar{x}_i^k d_i^k\right) \\
&= -\log\left(\omega\left(1 + \frac{1}{\omega} \alpha_k \bar{x}_i^k d_i^k\right)\right) \\
&= -\log \omega -\log\left(1 + \frac{1}{\omega} \alpha_k \bar{x}_i^k d_i^k\right) \\
&\geq -\log \omega -\frac{1}{\omega} \alpha_k \bar{x}_i^k d_i^k \\
&\geq -\log \omega - \frac{1}{\omega} \alpha_k |d_i^k|,
\end{align*}
where the second to last inequality follows by $\log(1+x) \leq x$ and
the last by $\bar{x}_i^k \leq 1$. Therefore, we have
\begin{equation} \label{eq:logdecreasexcross}
-\log(x_i^{k+1}) + \log(x_i^0) \geq -\log(\omega) + \log(x^0_i) - \frac{1}{\omega} \alpha_k |d_i^k|,
\end{equation}
so \eqref{eq:qk1} is satisfied again.
\item[3:] $x^{k+1}_i>\omega$ and $x^k_i> \omega$. For this case, we have
  \begin{align}
    \nonumber
-\log(x_i^{k+1}) = -\log\left(x_i^k + \alpha_k \bar{x}_i^k d_i^k\right)
&= -\log\left(x_i^k \left(1 + \alpha_k \frac{\bar{x}_i^k}{x_i^k} d_i^k\right)\right) \\
\nonumber
&= -\log(x_i^k)-\log\left(1 + \alpha_k \frac{\bar{x}_i^k}{x_i^k} d_i^k\right) \\
\nonumber
&\geq -\log(x_i^k)-\alpha_k \frac{\bar{x}_i^k}{x_i^k} d_i^k \\
\label{eq:wh8}
&\geq -\log(x_i^k)- \frac{1}{\omega} \alpha_k |d_i^k|,
\end{align}
where the second to last inequality follows by $\log(1+x) \leq x$ and
the last by $\bar{x}_i^k \leq 1$ and $x_i^k \geq \omega$.  We define
$\bar{k}$ to be the smallest index such that $x^j_i > \omega$ for all
$j=\bar{k}, \bar{k}+1, \dotsc, k+1$. We have that $\bar{k}$ exists,
and lies in the range $\{1,2,\dotsc,k\}$. Moreover, we have that
\begin{equation} \label{eq:wh9}
  x^{\bar{k}}_i > \omega, \quad x^{\bar{k}-1}_i \le \omega.
  \end{equation}
Since \eqref{eq:wh8} holds when $k$ is replaced by any
$j=\bar{k},\dotsc,k$, we have
\begin{equation} \label{eq:wh1}
  -\log x_i^{k+1} +  \log x^{\bar{k}}_i
  = \sum_{j=\bar{k}}^k \left( -\log x^{j+1}_i + \log x^j_i \right) \ge -\frac{1}{\omega} \sum_{j=\bar{k}}^k \alpha_j | d^j_i|.
  \end{equation}
Since $\bar{k}-1$ is in Case 2, because of \eqref{eq:wh9}, we have
\[
-\log x^{\bar{k}}_i \ge -\log \omega - \frac{1}{\omega} \alpha_{\bar{k}-1} |d^{\bar{k}-1}_i|.
\]
By adding this expression to \eqref{eq:wh1}, and adding $\log x^0_i$ to both sides, we obtain
\[
-\log x^{k+1}_i + \log x^0_i \ge - \log \omega + \log
x^0_i  - \frac{1}{\omega} \sum_{j=\bar{k}-1}^k \alpha_j |d^j_i|,
\]
which  implies \eqref{eq:qk1}.
    \end{itemize}
    }

\swmodify{
By summing \eqref{eq:qk1} over all coordinates $i$, we obtain
\begin{align*}
\sum_{i=1}^n \left(-\log(x^{k+1}_i) + \log(x^0_i)\right)
&\geq -\sum_{i=1}^n \left(\log(\omega)-\log(x_i^0)\right)
-\frac{1}{\omega} \sum_{j=0}^k \sum_{i=1}^n \alpha_j |d_i^j|\\
&= -\sum_{i=1}^n \left(\log(\omega)-\log(x_i^0)\right)
-\frac{1}{\omega} \sum_{j=0}^k \alpha_j \|d^j\|_1 \\
&\geq - n\left(\log(\omega)-\min_i \, \log(x_i^0)\right)
-\frac{\sqrt{n}}{\omega}\sum_{j=0}^k \alpha_j \|d^j\|,
\end{align*}
which proves the result.
}
\end{proof}

Now we are ready to bound the maximum number of iterations of
Algorithm~\ref{alg:lbncg} that can occur before the approximate
first-order optimality conditions \eqref{eq:xstrictfeas},
\eqref{eq:epsgradpos}, and \eqref{eq:epscompliment} are satisfied.
\begin{theorem} \label{thm:wcc1}
Let \refer{Assumptions~\ref{assum:fC22} and \ref{assum:flow}}
hold. Then, some iterate $x^k$ generated by Algorithm~\ref{alg:lbncg},
where $k=0,1,\dotsc,\bar{K}_1+1$ and
\refer{
\begin{align*}
\bar{K}_1 & := \left\lceil \frac{12\left(\mu n \left(\log(\omega_1)
-\min_i \log(x_i^0)\right) +f(x^0) - \flow\right)}
{\eta \call^3} 
\epsg^{-3/2} \right\rceil, \\
\omega_1 & := \max\left\{\frac{3 \sqrt{n}}{\eta \call^2}, \,
\|x^0\|_{\infty} \right\},\\
\call & := \min\{\csol,\cnc\},
\end{align*}}
will satisfy the conditions
\begin{equation} \label{eq:eps1ON}
\nabla f(x^k)  \geq -\epsg e, \quad
\|\bX_k \nabla f(x^k) \|_\infty \leq \epsg.
\end{equation}
\end{theorem}
\begin{proof}
Suppose for contradiction that at least one of the conditions in
\eqref{eq:eps1ON} is violated for all $k=0,1,\dotsc,\bar{K}_1+1$, so
that case A of Lemma~\ref{lem:LSccgSOL} does not occur for all
$k=0,1,\dotsc,\bar{K}_1$.  Algorithm~\ref{alg:ccg} will be invoked at
each of the first $\bar{K}_1+1$ iterates of
Algorithm~\ref{alg:lbncg}. For each iteration $l=0,1,\dotsc,\bar{K}_1$
for which Algorithm~\ref{alg:ccg} returns d\_type=SOL, we have from
Lemma~\ref{lem:LSccgSOL}, and the fact that case A does not occur,
that
\refer{$\alpha_k \|d^k\| \geq \csol \epsH$}.
For each iteration $l=0,1,\dotsc,\bar{K}_1$ for which
Algorithm~\ref{alg:ccg} returns d\_type=NC, we have by
 Lemma~\ref{lem:LSccgNC} that
\refer{$\alpha_k \|d^k\| \geq \cnc \epsH$}.
Thus, for either type of step, we have 
\refer{
\begin{equation} \label{eq:wcc1steplength}
\alpha_k \|d^k\| \geq \min\{\csol,\cnc\} \epsH = \call \epsH.
\end{equation}}

Now, by (\ref{eq:lsdecreasedamped}), we have
\[
-\frac{\eta}{6}\alpha_k^3 \|d^k\|^3 \geq \phi_\mu(x^{k+1}) - \phi_\mu(x^k)
= f(x^{k+1}) - f(x^k) + \mu \sum_{i=1}^n \left(-\log(x_i^{k+1}) + \log(x_i^k)\right). 
\]
By summing this bound over $k=0,1,\dotsc,\bar{K}_1$, and telescoping
both terms on the right-hand size, we obtain
\[
-\frac{\eta}{6}\sum_{k=0}^{\bar{K}_1} \alpha_k^3 \|d^k\|^3 \geq f(x^{\bar{K}_1+1}) - f(x^0)
+ \mu \sum_{i=1}^n \left(-\log(x_i^{\bar{K}_1+1}) + \log(x_i^0)\right). 
\]
\refer{By applying Lemma \ref{lem:logdecrease} with $\omega = \omega_1$, we
have
\begin{equation} \label{eq:wcc1eq1}
-\frac{\eta}{6}\sum_{k=0}^{\bar{K}_1} \alpha_k^3 \|d^k\|^3 \geq
f(x^{\bar{K}_1+1}) - f(x^0) - \mu n \left(\log(\omega_1)-\min_i \log(x^0_i)\right)
-\mu \frac{\sqrt{n}}{\omega_1} \sum_{k=0}^{\bar{K}_1} \alpha_k \|d^k\|. 
\end{equation}
From the definition of $\omega_1$, we obtain
\[
-\mu \frac{\sqrt{n}}{\omega_1} \sum_{k=0}^{\bar{K}_1} \alpha_k \|d^k\|
\geq -\frac{\mu \eta \call^2}{3} \sum_{k=0}^{\bar{K}_1} \alpha_k \|d^k\| 
= -\frac{\eta \call^2 \epsH^2}{12} \sum_{k=0}^{\bar{K}_1} \alpha_k \|d^k\|,
\]
where the final equality is due to $\mu = \epsg/4 = \epsH^2/4$. It
follows that
\begin{align*}
\frac{\eta}{6} \sum_{k=0}^{\bar{K}_1} \alpha_k^3 \|d^k\|^3 - \mu \frac{\sqrt{n}}{\omega_1}
\sum_{k=0}^{\bar{K}_1} \alpha_k \|d^k\|
&\geq\frac{\eta}{6} \sum_{k=0}^{\bar{K}_1} \alpha_k \|d^k\| \left(\alpha_k^2 \|d^k\|^2 
-\frac{\call^2 \epsH^2}{2}\right) \\
&\geq \frac{\eta}{12} \sum_{k=0}^{\bar{K}_1} \alpha_k \|d^k\| \call^2 \epsH^2\\
&\geq \frac{\eta}{12} \sum_{k=0}^{\bar{K}_1} \call^3 \epsH^3 \\
&= \frac{\eta}{12} \left( \bar{K}_1+1 \right) \call^3 \epsH^3,
\end{align*}
where the second and third inequalities follow by
\eqref{eq:wcc1steplength}.  By combining this inequality with
\eqref{eq:wcc1eq1}, we have
\begin{align*}
f(x^0) - f(x^{\bar{K}_1+1}) & +
\mu n \left(\log(\omega_1)-\min_i \log(x^0_i)\right) \\
&\geq  \left(\bar{K}_1+1\right)\frac{\eta}{12} \epsH^3 \call^3 \\
& > \mu n \left(\log(\omega_1)-\min_i \log(x^0_i)\right)  + f(x^0) - \flow,
\end{align*}
where we used the definition of $\bar{K}_1$ and $\epsH = \epsg^{1/2}$
for the final inequality. This inequality contradicts the definition
of $\flow$ \refer{(in Assumption \ref{assum:flow})}, so our claim is proved.}
\end{proof}

Recalling that the workload of Algorithm~\ref{alg:ccg} in terms of
Hessian-vector products depends on the index $J$ defined in
Lemma~\ref{lem:ccgits},
we obtain the following corollary. (Note the mild assumption on the
value of $M$ used at each instance of Algorithm~\ref{alg:ccg}, which
is satisfied provided that this algorithm is always invoked with an
initial estimate of $M$ in the range $[0,U_H+\mu]$.)
\begin{corollary} \label{coro:wcc1Hv}
Suppose that \refer{Assumptions \ref{assum:fC22}, \ref{assum:flow},
and \ref{assum:boundedgH}} hold, and let $\bar{K}_1$ be defined as in Theorem \ref{thm:wcc1}
and $J(M,\epsH,\zeta_r,c_\mu)$ be as defined \swmodify{in Lemma~\ref{lem:ccgits}.}
Suppose that the values of $M$ used or calculated at each instance of
Algorithm~\ref{alg:ccg} satisfy $M \le U_H+\mu$. Then the number of
Hessian-vector products and/or gradient evaluations required by
Algorithm~\ref{alg:lbncg} to output an iterate
satisfying~\eqref{eq:eps1ON} is at most
\begin{equation} \label{eq:ts2}
\left(2\min \left\{n,J(U_H+\mu,\epsH,\zeta_r,c_\mu) \right\}+2 \right)(\bar{K}_1+1). 
\end{equation}
\swmodify{
  If $J(U_H+\mu,\epsH,\zeta_r,c_\mu) < n$, this bound is
\begin{equation} \label{eq:ts3}
  \tcO (\epsg^{-7/4} + n\epsg^{-3/4} ),
\end{equation}
while if $J(U_H+\mu,\epsH,\zeta_r,c_\mu) \ge n$, it is
\begin{equation} \label{eq:ts4}
  \tcO (n \epsg^{-3/2} ).
\end{equation}
}
\end{corollary}
\begin{proof}
From Lemma~\ref{lem:ccgits}, the number of Hessian-vector
multiplications in the main loop of Algorithm~\ref{alg:ccg} is bounded
by $\min \left\{ n, J(U_H,\epsH,\zeta_r,c_\mu) \right\} + 1$.  An
additional $\min \left\{n, J(U_H,\epsH,\zeta_r,c_\mu) \right\}$
Hessian-vector products may be needed to return a direction
satisfying~\eqref{eq:weakcurvdir}, if Algorithm~\ref{alg:ccg} does not
store its iterates $y_j$. Each iteration also requires a single
evaluation of the gradient $\nabla f$, giving a bound of \\$(2 \min
\left\{n,J(U_H,\epsH,\zeta_r,c_\mu)\right\}+2)$ on the workload per
iteration of Algorithm~\ref{alg:lbncg}.  \swmodify{Per
  Theorem~\ref{thm:wcc1}, we obtain the result \eqref{eq:ts2} by
  multiplying this quantity by $\bar{K}_1+1$.}

\swmodify{ To obtain the estimate \eqref{eq:ts3}, we note from $\mu =
  \epsg/4$ that
  \[
  \bar{K}_1 = \tcO ( n \epsg^{-1/2} + \epsg^{-3/2} ),
  \]
  while from \eqref{eq:HvcountcappedCG1} and
  \eqref{eq:HvcountcappedCG}, using $\epsilon = \epsH = \epsg^{1/2}$,
  we have for $J(U_H+\mu,\epsH,\zeta_r,c_\mu) < n$ that
  \[
  J(U_H+\mu,\epsH,\zeta_r,c_\mu) = \tcO (\epsH^{-1/2} ) =
  \tcO (\epsg^{-1/4} ).
  \]
  We obtain \eqref{eq:ts3} by substituting these estimates into
  \eqref{eq:ts2}. For \eqref{eq:ts4}, we have from
  $J(U_H+\mu,\epsH,\zeta_r,c_\mu) \ge n$ together with
  \eqref{eq:HvcountcappedCG1} and \eqref{eq:HvcountcappedCG} that $n
  \le \tcO \left( \epsg^{-1/4} \right)$. Therefore,
  \swmodify{computational complexity} is bounded by
  \[
  \tcO ( n ( n \epsg^{-1/2} + \epsg^{-3/2} ))
  \le \tcO ( n (\epsg^{-3/4} + \epsg^{-3/2} ))
  = \tcO ( n \epsg^{-3/2} ),
  \]
  as claimed
}
\end{proof}

\subsection{Second-Order Complexity Analysis} \label{subsec:wcc2}

We now find bounds on iteration and computational complexity of
finding a point that satisfies all of the approximate optimality
conditions in \eqref{eq:epsKKT}.  In this section, as well as using
results from Sections~\ref{subsec:wccccg} and~\ref{subsec:wcc1}, we
need to use the properties of the minimum eigenvalue oracle,
Procedure~\ref{alg:meo}. To this end, we make the following generic
assumption.

\begin{assumption} \label{assum:wccmeo}
For every iteration $k$ at which Algorithm~\ref{alg:lbncg} calls
Procedure~\ref{alg:meo}, and for a specified failure probability
$\delta$ with $0 \le \delta \ll 1$, Procedure~\ref{alg:meo} either
certifies that $\bX_k \nabla^2 f(x_k) \bX_k \succeq -\epsH I$ or finds a vector
of curvature smaller than $-{\epsH}/{2}$ in at most
\begin{equation} \label{eq:wccmeo}
		N_{\mathrm{meo}}:=\min\left\{n,
		1+\left\lceil\Cmeo\epsH^{-1/2}\right\rceil\right\}
\end{equation}
Hessian-vector products, with probability $1-\delta$, where
$\Cmeo$ depends at most logarithmically on $\delta$ and $\epsH$.
\end{assumption}

Assumption~\ref{assum:wccmeo} encompasses the strategies we mentioned in
Section~\ref{subsec:meo}. Assuming the bound $U_H$ on
$\|H\|$ is available, for both the Lanczos method with a random
starting vector and the conjugate gradient algorithm with a random
right-hand side, \eqref{eq:wccmeo} holds with
$\Cmeo=\ln(2.75n/\delta^2)\sqrt{U_H}/2$. When a bound on
$\|H\|$ is not available in advance, it can be estimated efficiently
with minimal effect on the complexity bounds; see
Appendix B.3 of~\cite{CWRoyer_MONeill_SJWright_2019}.

The next lemma guarantees termination of the backtracking line search
for a negative curvature direction. As for Lemma~\ref{lem:LSccgSOL},
the result is deterministic.

\begin{lemma} \label{lem:LSgeneralNC}
Suppose that \refer{Assumptions~\ref{assum:fC22}
and~\ref{assum:wccmeo}} hold.
Suppose that at iteration $k$ of Algorithm~\ref{alg:lbncg}, the search
direction $d^k$ is of negative curvature type, obtained either directly
from Procedure~\ref{alg:meo} or as the output of Algorithm~\ref{alg:ccg}
with d\_type=NC. Then the backtracking line search terminates with step
length $\alpha_k = \theta^{j_k}$ with $j_k \le \jnc +1$, where $\jnc$ is defined
as in Lemma~\ref{lem:LSccgNC}, and the decrease in the function value
resulting from the chosen step length satisfies
\refer{
\begin{equation} \label{eq:LSgeneralNCdecrease}
\alpha_k \|d^k\| \geq \frac14 \cnc \epsH,
\end{equation}
}with $\cnc$ is defined in Lemma~\ref{lem:LSccgNC}.
\end{lemma}
\begin{proof}
Lemma~\ref{lem:LSccgNC} shows that the claim holds \swmodify{(with a
  factor of $1/4$ to spare)} when the direction of negative curvature
is obtained from Algorithm~\ref{alg:ccg}.  When the direction $v$ is
obtained from Procedure~\ref{alg:meo}, we have by $\|v\| = 1$ that
\[
v^\top \bX_k \nabla^2 f(x^k) \bX_k v \leq -\frac12 \epsH.
\]
Then, since $v^\top \bX_k X^{-2}_k \bX_k v \leq 1$, we have
\begin{align}
  \nonumber
v^\top \bX_k \nabla^2 \phi_\mu(x^k) \bX_k v & =
v^\top \bX_k \nabla^2 f(x^k) \bX_k v + \mu v^\top \bX_k X^{-2}_k \bX_k v \\
\label{eq:LSgeneralNCeq1}
& \leq 
-\frac12 \epsH + \mu \le -\frac14 \epsH,
\end{align}
where the last inequality follows from $\mu = \epsg/4 = \epsH^2/4$ and
$\epsH < 1$.
Now, when
\[
\min\left\{|v^\top \bX_k \nabla^2 \phi_\mu(x^k) \bX_k v|,
\frac{\beta}{\|X_k^{-1} \bX_k v\|_\infty}\right\} =
|v^\top \bX_k \nabla^2 \phi_\mu(x^k) \bX_k v|,
\]
we have $\|d^k\| = |v^\top \bX_k \nabla^2 \phi_\mu(x^k) \bX_k
v| \ge \epsH/4$. Otherwise, we have
\[
\beta = \|X_k^{-1} \bX_k d^k\|_\infty \leq \|X_k^{-1} \bX_k d^k\| \leq
\|X_k^{-1} \bX_k\| \|d^k\| \leq \|d^k\|.
\]
By combining the two cases, and using $\beta \ge \epsH$, we have
\[
\|d^k\| \geq  \min\left\{\frac14 \epsH, \beta\right\} = \frac14 \epsH.
\]
Finally, we note that in either case, we have
\[
\|d^k\| \leq -v^\top \bX_k \nabla^2 \phi_\mu(x^k) \bX_k v =
-\frac{(d^k)^\top \bX_k \nabla^2 \phi_\mu(x^k) \bX_k d^k}{\|d^k\|^2}.
\]
Therefore, we have
\[
\frac{(d^k)^\top \bX_k \nabla^2 \phi_\mu(x^k) \bX_k d^k}{\|d^k\|^2} \leq - \|d^k\| \leq -\frac14 \epsH.
\]
The result can now be obtained by following the proof of
Lemma~\ref{lem:LSccgNC}, with $\frac14 \epsH$ replacing $\epsH$.
\end{proof}

We are now ready to state our iteration complexity result for
Algorithm~\ref{alg:lbncg}.
\begin{theorem} \label{thm:wccits}
Suppose that \refer{Assumptions~\ref{assum:fC22},~\ref{assum:flow},
and~\ref{assum:wccmeo}} hold and define
\refer{
\begin{equation} \label{eq:omega2def}
\omega_2 := \max\left\{\frac{96 \sqrt{n}}{\eta \call^2}, \, \| x^0\|_{\infty} \right\},
\end{equation}
and
\begin{align} \label{eq:wccits}
\bar{K}_2 & :=\left\lceil \frac{1536\left(f(x^0) - \flow
+\mu n \left(\log(\omega_2)-\min_i \log(x^0_i)\right)\right)}
    {\eta \call^3} \epsg^{-3/2} \right\rceil + 2 \\
    \nonumber
    & = \tcO(n \epsg^{-1/2} + \epsg^{-3/2}),
\end{align}
where the constant $\call$ is defined in Theorem~\ref{thm:wcc1}.
}Then with probability at least
$(1-\delta)^{\bar{K}_2}$, Algorithm~\ref{alg:lbncg} terminates
at a point satisfying \eqref{eq:epsKKT} in at most $\bar{K}_2$
iterations. (With probability at most
$1-(1-\delta)^{\bar{K}_2}$, it terminates incorrectly within
$\bar{K}_2$ iterations at a point for which (\ref{eq:xstrictfeas}), (\ref{eq:epsgradpos}),
and (\ref{eq:epscompliment}) hold but (\ref{eq:epspsd}) does not.)
\end{theorem}
\begin{proof}
Algorithm~\ref{alg:lbncg} terminates incorrectly with probability
$\delta$ at any iteration at which Procedure~\ref{alg:meo} is called,
when Procedure~\ref{alg:meo} certifies erroneously that \\
$\lambda_{\min}(\bX_k \nabla^2 f(x^k) \bX_k) \geq -\epsH$.
Such an erroneous certificate only leads to termination.
Therefore, an erroneous certificate at iteration $k$ means that
Procedure~\ref{alg:meo} did not produce an erroneous certificate at
iterations $0$ to $k-1$. By a disjunction argument, we have that the
overall probability of terminating with an erroneous certificate
during the first $\bar{K}_2$ iterations is bounded by
$1-(1-\delta)^{\bar{K}_2}$. Therefore, with probability at least
$(1-\delta)^{\bar{K}_2} $, no incorrect termination occurs in the
first $\bar{K}_2$ iterations.

Suppose now for contradiction that Algorithm~\ref{alg:lbncg} runs for
$\bar{K}_2$ iterations without terminating. That is, for all 
$l=0,1,\dotsc,\bar{K}_2$, we have at least one of:
$[\nabla f(x^l)]_i <  -\epsg$ for some coordinate $i$,
$\| \bX_l \nabla f(x^l) \|_\infty > \epsg$,
or $\lambdamin(\bX_l \nabla^2 f (x^l) \bX_l) < -\epsH$.
Consider the following  partition of the set of iteration indices:
\begin{equation} \label{eq:K123}
  \cK_1 \cup \cK_2 \cup \cK_3 = \{0,1,\dotsc,\bar{K}_2-1\},
  \end{equation}
where $\cK_1$, $\cK_2$, and $\cK_3$ are defined as follows.

\textbf{Case 1:} $\cK_1 := \{ l =0,1,\dotsc,\bar{K}_2-1 \, : \,
  \nabla f(x^l) \geq -\epsg e$ and
  $\| \bX_l \nabla f(x^l) \|_\infty \leq \epsg \}$.
  
  

\refer{\textbf{Case 2:} $\cK_2 := \{ l=0,1,\dotsc,\bar{K}_2-1 \, : \, [\nabla
  f(x^l)]_i < -\epsg$ for some coordinate $i$ and/or $\| \bX_l \nabla
  f(x^l) \|_\infty > \epsg$ and $\alpha_l \|d^l\| \geq  (\call/4) \epsH \}$.}
  
\refer{\textbf{Case 3:} $\cK_3 := \{ l=0,1,\dotsc,\bar{K}_2-1 \, : \,
  [\nabla f(x^l)]_i <  -\epsg$ for some coordinate $i$ and/or
  $\| \bX_l \nabla f(x^l) \|_\infty > \epsg$ and $\alpha_l \|d^l\| < (\call/4) \epsH \}$.}
  
  \refer{Then, for all $l \in \cK_1 \cup \cK_2$, the fact that the algorithm
  does not satisfy \eqref{eq:epsKKT} at iteration $l+1$ together with
  Lemmas~\ref{lem:LSccgSOL}, \ref{lem:LSccgNC}, and
  \ref{lem:LSgeneralNC} guarantee that
\begin{equation} \label{eq:wcc2steplength}
\alpha_l \|d^l\| \geq  \min\{\csol,\cnc/4\} \epsH  \ge (\call/4) \epsH.
\end{equation}
On the other hand, for $l \in \cK_3$, case A of Lemma
\ref{lem:LSccgSOL} must have occured. Therefore, for any $l \in
\cK_3$, we must have $\nabla f(x^{l+1}) \geq -\epsg e$ and $\|
\bX_{l+1} \nabla f(x^{l+1}) \|_\infty \leq \epsg$, so that $l+1 \in
\cK_1$ for $l<\bar{K}_2-1$. Thus, a sufficiently long step will be
taken at the {\em next} iteration, and we have
\begin{equation} \label{eq:K13}
  | \cK_3| \leq |\cK_1|+1 \leq |\cK_1| + |\cK_2| + 1.
\end{equation}}

\refer{Now, by a similar argument to Theorem \ref{thm:wcc1}
that led to (\ref{eq:wcc1eq1}), we have
\begin{equation} \label{eq:wcc2eq1}
-\frac{\eta}{6} \sum_{j=0}^{\bar{K}_2-1} \alpha_j^3 \|d^j\|^3 \geq
f(x^{\bar{K}_2}) - f(x^0) -\mu\frac{\sqrt{n}}{\omega_2}
\sum_{l=0}^{\bar{K}_2-1} \alpha_l \|d^l\| -\mu
n\left(\log(\omega_2)-\min_i \log(x_i^0)\right).
\end{equation}}

\refer{Using the definition of $\omega_2$, we have
\[
-\mu \frac{\sqrt{n}}{\omega_2} \sum_{l=0}^{\bar{K}_2-1} \alpha_l
\|d^l\| \geq -\frac{\mu \eta \call^2}{96} \sum_{l=0}^{\bar{K}_2-1}
\alpha_l \|d^l\| = -\frac{\eta \call^2 \epsH^2}{384}
\sum_{l=0}^{\bar{K}_2-1} \alpha_l \|d^l\|,
\]
where the second equality is due to $\mu = \epsg/4 = \epsH^2/4$.
Therefore, we have
\begin{align*}
& \frac{\eta}{6} \sum_{l=0}^{\bar{K}_2-1} \alpha_l^3 \|d^l\|^3
 - \mu \frac{\sqrt{n}}{\omega_2} \sum_{l=0}^{\bar{K}_2-1} \alpha_l \|d^l\| \\
& \geq\frac{\eta}{6} \sum_{l=0}^{\bar{K}_2-1}\left(\alpha_l^3 \|d^l\|^3 
-\frac{\call^2 \epsH^2}{64}\alpha_l \|d^l\|\right) \\
&= \frac{\eta}{6} \sum_{j \in \cK_1 \cup \cK_2} \alpha_j \|d^j\| \left(\alpha_j^2 \|d^j\|^2
- \frac{\call^2 \epsH^2}{64}\right)
+ \frac{\eta}{6} \sum_{l \in \cK_3} \left(\alpha_l^3 \|d^l\|^3
- \frac{\call^2 \epsH^2}{64}\alpha_l \|d^l\|\right) \\
&\geq \frac{3\eta}{384} \sum_{j \in \cK_1 \cup \cK_2} \alpha_j \|d^j\| \call^2 \epsH^2
- \frac{\eta}{1536} \sum_{l \in \cK_3} \call^3 \epsH^3 \\
&\geq (|\cK_1| +|\cK_2|) \frac{3\eta}{1536} \call^3 \epsH^3
- \left(|\cK_1| + |\cK_2| + 1\right) \frac{\eta}{1536} \call^3 \epsH^3 \\
&\geq \left(|\cK_1| +|\cK_2| -\frac12\right) \frac{\eta}{768} \call^3 \epsH^3,
\end{align*}
where the second inequality follows by (\ref{eq:wcc2steplength}) and
the definition of $\cK_3$, while the third inequalities follows by
(\ref{eq:wcc2steplength}) and (\ref{eq:K13}).}

\refer{Thus, this inequality, (\ref{eq:wcc2eq1}) and
$|\cK_1| + |\cK_2| + |\cK_3| - 2 \leq 2(|\cK_1| + |\cK_2| - 1/2)$, imply
\begin{align*}
& f(x^0) - f(x^{\bar{K}_2}) +
\mu n \left(\log(\omega_2)-\min_i \log(x^0_i)\right) \\
& \geq \left(|\cK_1| +|\cK_2| - 1/2\right) \frac{\eta}{768} \call^3 \epsH^3 \\
&\geq \left(|\cK_1| +|\cK_2| + |\cK_3|-2\right) \frac{\eta}{1536} \call^3 \epsH^3 \\
&\geq \left(\bar{K}_2-1\right) \frac{\eta}{1536} \call^3 \epsH^3 \\
&>  f(x^0) - \flow
 + \mu n \left(\log(\omega_2)-\min_i \log(x^0_i)\right).
\end{align*}
where the final inequality follows from the definition of $\bar{K}_2$
and $\epsH = \epsg^{1/2}$. The final inequality implies that $\flow >
f(x^{\bar{K}_2})$, which contradicts the definition of $\flow$,
proving the claim.}

\refer{The estimate $\bar{K}_2 = \tcO(n \epsg^{-1/2} + \epsg^{-3/2})$ follows
directly from $\mu = \epsg/4$.}
\end{proof}

Finally, we provide a computational complexity result, a bound on the
number of Hessian-vector products and gradient evaluations necessary
for Algorithm~\ref{alg:lbncg} to find a point that satisfies
(\ref{eq:epsKKT}).
\begin{corollary} \label{coro:wcc2Hv}
Suppose that \refer{Assumptions~\ref{assum:fC22},~\ref{assum:flow},
\ref{assum:boundedgH}, and~\ref{assum:wccmeo}} hold, and let
$\bar{K}_2$ be defined as in \eqref{eq:wccits}.  Suppose that the
values of $M$ used or calculated at each instance of
Algorithm~\ref{alg:ccg} satisfy $M \le U_H+\mu$. Then with probability at
least $(1-\delta)^{\bar{K}_2}$, Algorithm~\ref{alg:lbncg} terminates
at a point satisfying \eqref{eq:epsKKT} after at most
\begin{equation} \label{eq:qj8}
\left(\max \{ 2\min\{n,J(U_H+\mu,\epsH,\zeta_r,c_\mu)\}+2,N_{\mathrm{meo}} \} \right) 
\bar{K}_2
\end{equation}
Hessian-vector products and/or gradient evaluations. (With probability
at most $1-(1-\delta)^{\bar{K}_2}$, it terminates incorrectly with
this complexity at a point \refer{for which} (\ref{eq:xstrictfeas}),
(\ref{eq:epsgradpos}), and (\ref{eq:epscompliment}) hold but
(\ref{eq:epspsd}) does not.)
  
\end{corollary}
\begin{proof}
The proof follows by combining Theorem~\ref{thm:wccits} (which bounds
the number of iterations) with Lemma~\ref{lem:ccgits} and
Assumption~\ref{assum:wccmeo} (which bound the workload per
iteration).
\end{proof}

\sjwresolved{Moved these ``loose'' statements outside of the formal statement
  of the corollary. Are they correct?} \monresolved{These appear to be correct.}
  \swmodify{For large $n$, the
  operation bound \eqref{eq:qj8} is $\tcO( \epsg^{-7/4} + n
  \epsg^{-3/4} )$, because the multiplier of $\bar{K}_2$ in
  \eqref{eq:qj8} is $\tcO(\epsg^{-1/4})$ while $\bar{K}_2$ is $\tcO(n
  \epsg^{-1/2} + \epsg^{-3/2})$. For small $n$, the multiplier of
  $\bar{K}_2$ in \eqref{eq:qj8} is $\cO(n)$, and the dominant term in
  $\bar{K}_2$ is $\epsg^{-3/2}$, leading to a computational complexity
  bound of $\tcO(n \epsg^{-3/2})$ for this case.}

\swmodify{These computational complexity bounds are the same as those
  obtained for {\em unconstrained} smooth minimization discussed in
  Section~\ref{sec:relatedworks}, except for the inclusion of the $n
  \epsg^{-3/4}$ term for the case of large $n$. In the latter case,
  our algorithm acheives a superior worst-case computational
  complexity bound to that of \cite{GHaeser_HLiu_YYe_2018}, whose
  worst-case computational complexity appear to be $\cO(n
  \epsg^{-3/2})$. The $n \epsg^{-3/4}$ term is a consequence
  of using the log-barrier term to
  monitor descent. It may be avoided by making an additional
  assumption that $f$ grows rapidly enough to overcome the improvement
  in the logarithmic term of $\phi_\mu$, as $x$ moves away from the
  solution set for \eqref{eq:fdef} and becomes large. Indeed, we made
  such an assumption in an earlier version of the paper. It makes the
  analysis somewhat more straightforward in that it allows us assume
  that the iterates $\{ x^k \}$ are bounded. However, prompted by a
  referee's comment and a desire for generality, we have dropped this
  assumption in the current version.}


\section{Discussion} \label{sec:discussion}
We have presented a log-barrier Newton-CG algorithm which combines
recent advances in complexity of algorithms for large-scale
unconstrained optimization with results on the primal log-barrier
function for bound constraints. Our algorithm uses the Capped CG
method of \cite{CWRoyer_MONeill_SJWright_2019} to compute Newton-type
steps for the log-barrier function, while monitoring convexity during
the CG iterations to detect possible directions of negative curvature.
Once the algorithm has found a point satisfying the first-order
optimality conditions, a Minimum Eigenvalue Oracle is used to find a
direction of negative curvature for the scaled Hessian matrix or to
certify (with high probability) that the second-order optimality
conditions hold at the current iterate.  Both types of steps can be
computed using efficient iterative solvers, enabling good overall
computational complexity results. The resulting method finds a point
satisfying (\ref{eq:epsKKT}) in at most
\refer{$\cO(\epsg^{-3/2}+n\epsg^{-1/2})$} iterations, with at most
\refer{$\tcO(n \epsg^{-3/2})$ gradient evaluations and/or Hessian
  vector products when $n$ is small and at most $\tcO(\epsg^{-7/4} + n
  \epsg^{-3/4})$ gradient evaluations and/or Hessian vector products
  for $n$ sufficiently large. This overall computational complexity
  compares favorably with the worst-case bounds of recently proposed
  methods.}


There are a number of ways to align our algorithm more closely with
the interior-point methods in common use. One possible extension is to
embed this method in a primal-dual interior-point framework, which is
more \swmodify{widely used} than the primal log-barrier framework. A
second is to extend the log-barrier approach to minimize $\phi_\mu$
for a decreasing positive sequence of values of $\mu$, rather than the
``one-shot'' approach using a small fixed value of $\mu$ that we
describe in this paper. Finally, generalizations of our approach to
problems with more complex constraint sets, such as problems with
general linear constraints, remains an open problem.

\section*{Acknowledgement}
We thank Cl\'ement Royer for his valuable advice and many suggestions
during the preparation of this manuscript. We are also grateful for
the very helpful comments of the associate editor and two referees on
an earlier version.

\bibliographystyle{siam} 
\bibliography{refs-newtonlanczoscg}

\appendix
\section{Proofs of Technical Results} \label{app:proof}

\subsection{Proof of Lemma~\ref{lem:logbound}.} \label{app:proof.logbound}

\begin{proof}
For scalar $y>-1$, define $g(y) = -\log(1+y)$. We have $g'(y) =
-1/(1+y)$, $g''(y) = 1/(1+y)^2$ and $g^{(3)}(y) = -2/(1+y)^3$.  By
Taylor's theorem, we have
\begin{equation} \label{eq:logboundeq1}
g(y) = g(0) + y g'(0) + \frac12 y^2 g''(0) + \frac12 \int_0^y (y-t)^2 g^{(3)}(t) dt.
\end{equation}
Substituting $t = yu$ and using $|y| \leq \beta < 1$, we have
\[
\frac12 \int_0^y (y-t)^2 g^{(3)}(t) dt = -y \int_0^1 (y-yu)^2 \frac{du}{(1+yu)^3}
\leq |y|^3 \int_0^1 (1 - u)^2 \frac{du}{(1-\beta u)^3}.
\]
Now, since $(1 - u)^2$ is monotonically decreasing in $u$ and
$1/(1-\beta u)^3$ is monotonically increasing in $u$, we can apply
Chebyshev's integral inequality:
\begin{align*}
|y|^3 \int_0^1 (1 - u)^2 \frac{du}{(1-\beta u)^3}
&\leq |y|^3 \left[\int_0^1 (1 - u)^2  du \right] \left[ \int_0^1 \frac{du}{(1-\beta u)^3} \right] \\
&=  \frac{|y|^3}{3} \left[ \int_0^1 \frac{du}{(1-\beta u)^3} \right] \\
&= \frac{|y|^3}{6} \frac{1}{\beta} \left(\frac{1}{(1-\beta)^2} - 1\right) 
= \frac{|y|^3}{6} \frac{2 - \beta}{(1-\beta)^2}.
\end{align*}
By combining with (\ref{eq:logboundeq1}), we obtain
\[
-\log(1+y) \leq -y + \frac12 y^2 + \frac{|y|^3}{6} \frac{2 - \beta}{(1-\beta)^2}.
\]
Now, for some coordinate $i$, let $y = \left(\bar{x}_i/x_i\right) d_i$.
Clearly, we have $|y| \leq \beta$ so
\begin{align}
  \nonumber
-\log\left(1 + \frac{\bar{x}_i}{x_i} d_i\right)
&\leq -\frac{\bar{x}_i}{x_i} d_i + \frac12 \left(\frac{\bar{x}_i}{x_i} d_i \right)^2
+ \frac{| \frac{\bar{x}_i}{x_i} d_i|^3}{6} \frac{2 - \beta}{(1-\beta)^2} \\
\label{eq:su3}
&\leq -\frac{\bar{x}_i}{x_i} d_i + \frac12 \left(\frac{\bar{x}_i}{x_i} d_i \right)^2
+ \frac{|d_i|^3}{6} \frac{2 - \beta}{(1-\beta)^2}
\end{align}
holds. By the properties of logarithms, we have
\[
-\log \left(x_i \left(1 + \frac{\bar{x}_i}{x_i} d_i \right) \right) =-\log (x_i + \bar{x}_i d_i) =
-\log(x_i) - \log \left( 1 + \frac{\bar{x}_i}{x_i} d_i  \right).
\]
By rearranging this inequality and substituting from \eqref{eq:su3},
we have
\[
-\log \left(x_i + \bar{x}_id_i \right) + \log(x_i) \leq -\frac{\bar{x}_i}{x_i} d_i
+ \frac12 \left(\frac{\bar{x}_i}{x_i} d_i \right)^2
+ \frac{|d_i|^3}{6} \frac{2 - \beta}{(1-\beta)^2}.
\]
By summing this inequality over $i=1,2,\dotsc, n$, we obtain
\begin{align*}
&-\sum_{i=1}^n \log \left(x_i + \bar{x}_i d_i \right) + \sum_{i=1}^n \log(x_i) \\
&\leq -e^\top X^{-1} \bX d + \frac12 d^\top \bX X^{-2} \bX d
+ \sum_{i=1}^n \frac{|d_i|^3}{6} \frac{2 - \beta}{(1-\beta)^2} \\
&= -e^\top X^{-1} \bX d + \frac12 d^\top \bX X^{-2} \bX d
+ \frac{2 - \beta}{6(1-\beta)^2}\|d\|_3^3  \\
&\leq -e^\top X^{-1} \bX d + \frac12 d^\top \bX X^{-2} \bX d
+ \frac{2 - \beta}{6(1-\beta)^2}\|d\|^3,
\end{align*}
where $\| d \|_3$ denotes the $\ell_3$ norm of $d$. (The final inequality
follows from $\|d\|_3 \leq \|d\|_2$).
\end{proof}

\subsection{Proof of Lemma~\ref{lem:LSccgSOL}.} \label{app:proof.LSccgSOL}
\begin{proof}
For simplicity of notation, we again use $H = \bX_k \nabla^2 \phi_\mu(x^k) \bX_k$
and $g = \bX_k \nabla \phi_\mu(x^k)$ in the proof.

Suppose first that the unit step length $\alpha_k = 1$ is accepted.
Then, if $\|d^k\| < c_d \epsH$, it follows from Lemma
\ref{lem:shortdsmallKKT} that both (\ref{eq:epsgradpos}) and
(\ref{eq:epscompliment}) hold at $x^{k+1}$, so we are in case
A. Otherwise, the statment of case B holds by
\[
\alpha_k \|d^k\| = \|d^k\| \geq c_d \epsH \geq \csol \epsH.
\]

For the remainder of the proof, we assume that $\alpha_k < 1$.  Recall from
the statement of Lemma~\ref{lem:ccgsteps} that
\[
\gamma_k = \max\left\{\frac{\|X_k^{-1} \bX_k
  \hat{d}^k\|_\infty}{\beta}, 1 \right\}.
\]
For any $j \geq 0$ such that the sufficient decrease condition
(\ref{eq:lsdecreasedamped}) does not hold, we have from
(\ref{eq:f2smooth}), (\ref{eq:ccgstepSOLcurv}),
(\ref{eq:ccgstepSOLgradcond}), and Lemma~\ref{lem:logbound} that
\begin{align*}
& -\frac{\eta}{6} \theta^{3j} \|d^k\|^3 \\
&\leq \phi_\mu(x^k + \theta^j \bX_k d^k) - \phi_\mu(x^k)\\
&\leq \theta^j \nabla f(x^k)^\top \bX_k d^k
+ \frac{\theta^{2j}}{2} (d^k)^\top \bX_k \nabla^2 f(x^k) \bX_k d^k
+ \frac{L_H}{6} \theta^{3j}\|\bX_k d^k\|^3 && \text{by (\ref{eq:f2smooth})} \\
&\quad- \mu \theta^j e^\top X_k^{-1} \bX_k d^k
+ \frac{\mu \theta^{2j}}{2} (d^k)^\top \bX_k X_k^{-2} \bX_k d^k
+ \frac{\mu (2 - \beta)}{6(1-\beta)^2} \theta^{3j} \|d^k\|^3 && \text{by Lemma \ref{lem:logbound}} \\
&= \theta^j g^\top d^k
+ \frac{\theta^{2j}}{2} (d^k)^\top H d^k
+ \frac{L_H}{6} \theta^{3j}\|\bX_k d^k\|^3
+ \frac{\mu (2 - \beta)}{6(1-\beta)^2} \theta^{3j} \|d^k\|^3 \\
&= -\theta^j \gamma_k (d^k)^\top \left(H + 2\epsH I\right) d^k 
+ \frac{\theta^{2j}}{2} (d^k)^\top H d^k  && \text{by (\ref{eq:ccgstepSOLgradcond})} \\
&\quad+ \frac{L_H}{6} \theta^{3j}\|\bX_k d^k\|^3
+ \frac{\mu (2 - \beta)}{6(1-\beta)^2} \theta^{3j} \|d^k\|^3  \\
&= -\theta^j \left(\gamma_k - \frac{\theta^j}{2}\right)
(d^k)^\top \left(H + 2\epsH I \right) d^k - \theta^{2j} \epsH \|d^k\|^2 \\
&\quad+ \frac{L_H}{6} \theta^{3j}\|\bX_k d^k\|^3
+ \frac{\mu (2 - \beta)}{6(1-\beta)^2} \theta^{3j} \|d^k\|^3 \\
&\leq -\theta^j \gamma_k \epsH \|d^k\|^2
+ \frac12 \theta^{2j} \epsH \|d^k\|^2 - \theta^{2j} \epsH \|d^k\|^2
&& \text{by (\ref{eq:ccgstepSOLcurv})} \\
&\quad+ \frac{ L_H (1-\beta)^2 + (2 - \beta)}{6(1-\beta)^2} \theta^{3j}\|d^k\|^3 && \text{by $\mu<1$} \\
&\leq -\theta^j \gamma_k \epsH \|d^k\|^2
+ \frac{ L_H (1-\beta)^2 + (2 - \beta)}{6(1-\beta)^2} \theta^{3j}\|d^k\|^3.
\end{align*}
Therefore, for any $j \geq 0$ at which sufficient decrease is not
attained, we have by rearranging terms in the inequality above and
using the definition of $\gamma_k$ that
\begin{align} \nonumber
\frac{ (L_H + \eta)(1-\beta)^2 + (2 - \beta)}{6(1-\beta)^2} \theta^{2j}
& \geq \max\left\{\frac{\|X_k^{-1} \bX_k \hat{d}^k\|_\infty}{\beta}, 1\right\} \epsH \|d^k\|^{-1} \\
 \label{eq:sollseq1}
& \geq \epsH \|d^k\|^{-1}.
\end{align}
Evaluating this expression at $j = 0$, we have that
\begin{equation} \label{eq:sollseq2}
\|d^k\| \geq \frac{6(1-\beta)^2}{(L_H + \eta)(1-\beta)^2 + (2 - \beta)} \epsH.
\end{equation}
From (\ref{eq:ccgstepSOLnorm}), we have
\[
\|d^k\| \leq 1.1 \epsH^{-1} \|g\| \leq 1.1 \epsH^{-1} (\|\bX_k \nabla f(x^k)\| + \mu \|\bX_k X_k^{-1} e\|)
\leq 1.1 \epsH^{-1} (U_g + \mu \sqrt{n}),
\]
where we used $\|\bX_k\| \leq 1$, $\|\nabla f(x^k)\| \leq U_g$, and
$\|\bX_k X_k^{-1} e\| \leq \sqrt{n}$ in the final inequality. Thus,
for any $j > \jsol$ we have from definition (\ref{eq:ccglsitsSOL}) and
this bound on $\|d^k\|$ that
\begin{align*}
\theta^{2j} < \theta^{2\jsol} &\leq \frac{6(1-\beta)^2}
{(L_H + \eta)(1-\beta)^2 + (2 - \beta)} \frac{\epsH^2}{1.1(U_g + \mu \sqrt{n})} \\
&\leq \frac{6(1-\beta)^2 \epsH}{(L_H + \eta)(1-\beta)^2 + (2 - \beta)}\|d^k\|^{-1}.
\end{align*}
Therefore, (\ref{eq:sollseq1}) cannot be satisfied for any $j > \jsol$ so the line search
must terminate with $\alpha_k = \theta^{j_k}$ for some $1 \leq j_k \leq \jsol + 1$.
The previous index $j_k-1$ satisfies (\ref{eq:sollseq1}), so we  also have
\[
\theta^{2(j_k-1)} = \frac{\theta^{2j_k}}{\theta^2}
\geq \frac{6(1-\beta)^2  \epsH}
{(L_H + \eta)(1-\beta)^2 + (2 - \beta)}\|d^k\|^{-1}.
\]
It follows that
\begin{align*}
\alpha_k \|d^k\| = \theta^{j_k} \|d^k\|
&\geq \left(\frac{6(1-\beta)^2\theta^2  \epsH}
{(L_H + \eta)(1-\beta)^2 + (2 - \beta)}\right)^{1/2} \|d^k\|^{1/2} \\
&\geq \frac{6(1-\beta)^2\theta^2}
{(L_H + \eta)(1-\beta)^2 + (2 - \beta)} \epsH
\end{align*}
holds, where the final inequality comes from (\ref{eq:sollseq2}) and $\theta < 1$.
Thus, the conclusion holds in this case as well and the proof is complete.
\end{proof}

\subsection{Proof of Lemma~\ref{lem:LSccgNC}.} \label{app:proof.LSccgNC}
\begin{proof}
We again use the notation \\$H = \bX_k \nabla^2 \phi_\mu(x^k) \bX_k$ and
$g = \bX_k \nabla \phi_\mu(x^k)$ in this proof.

We begin by noting that when the unit step, $\alpha_k = 1$, is taken, we have
\[
\alpha_k \|d^k\| = \|d^k\| \geq \epsH,
\]
where the inequality follows from (\ref{eq:ccgstepNC}).

In the remainder of the proof, we assume that the unit step length is
not accepted. Then, for any $j \geq 0$ such that
(\ref{eq:lsdecreasedamped}) does not hold, we have from
(\ref{eq:f2smooth}) and (\ref{eq:ccgstepNC}) along with the result of
Lemma \ref{lem:logbound} that
\begin{align*}
& -\frac{\eta}{6}\theta^{3j} \|d^k\|^3 \\
&\leq \phi_\mu(x^k + \theta^j \bX_k d^k) - \phi_\mu (x^k) \\
&\leq \theta^j \nabla f(x^k)^\top \bX_k d^k
+ \frac{\theta^{2j}}{2} (d^k)^\top \bX_k \nabla^2 f(x^k) \bX_k d^k
+ \frac{L_H}{6} \theta^{3j}\|\bX_k d^k\|^3 && \text{by (\ref{eq:f2smooth})}\\
&\quad - \mu \theta^j e^\top X_k^{-1} \bX_k d^k
+ \frac{\mu \theta^{2j}}{2} (d^k)^\top \bX_k X_k^{-2} \bX_k d^k
+ \frac{\mu (2 - \beta)}{6(1-\beta)^2} \theta^{3j} \|d^k\|^3  && \text{by Lemma \ref{lem:logbound}}\\
&= \theta^j g^\top d^k + \frac{\theta^{2j}}{2} (d^k)^\top H d^k
+ \frac{L_H}{6} \theta^{3j}\|\bX_k d^k\|^3
+ \frac{\mu (2 - \beta)}{6(1-\beta)^2} \theta^{3j} \|d^k\|^3  \\
&\leq -\frac{\theta^{2j}}{2} \|d^k\|^3
+\frac{L_H (1-\beta)^2 + (2 - \beta)}{6(1-\beta)^2} \theta^{3j}\|d^k\|^3,
\quad \text{by (\ref{eq:ccgstepNC}) and $\mu < 1$.}
\end{align*}
By rearranging this
expression, we have for all such $j$ that
\[
\theta^j \geq \frac{3(1-\beta)^2}{(L_H + \eta) (1-\beta)^2 + (2 - \beta)}
\]
which is true only for $j \leq \jnc$. Thus, the line search must terminate
for some $j_k \leq \jnc + 1$. Since the line search failed to stop at
iteration $j_k-1$, we must have
\[
\theta^{j_k-1} = \frac{\theta^{j_k}}{\theta} \geq \frac{3(1-\beta)^2 }{(L_H + \eta) (1-\beta)^2 + (2 - \beta)}.
\]
Therefore, using $\|d^k\| \geq \epsH$ from (\ref{eq:ccgstepNC}), we have that
\[
\alpha_k \|d^k\| = \theta^{j_k} \|d^k\| \geq
\frac{3(1-\beta)^2 \theta}{(L_H + \eta) (1-\beta)^2 + (2 - \beta)} \epsH
\]
as required.
\end{proof}

\end{document}